\documentclass{article}
\usepackage[utf8]{inputenc}
\usepackage{geometry}
\usepackage{amsmath}
\usepackage{amssymb}

\usepackage{bm}

\usepackage{graphicx}

\usepackage{subcaption}

\usepackage[hidelinks]{hyperref} % hidelinks option removes the red/green boxes

\usepackage{amsthm} %Thm environment
\newtheorem{definition}{Definition}
\newtheorem{example}{Example}
\newtheorem{lemma}{Lemma}
\newtheorem{theorem}{Theorem}

\newtheorem{corollary}{Corollary}
\newtheorem{remark}{Remark}
\newtheorem{proposition}{Proposition}

\usepackage{bm}

\usepackage{subcaption}

\usepackage[affil-sl]{authblk}

\begin{document}
	
	\title{Convergence analysis of oversampled collocation boundary element methods in 2D\thanks{The work of the first author was supported by the UK Engineering and Physical Sciences Research Council (EPSRC) grant EP/L016516/1 for the University of Cambridge Centre for Doctoral Training, the Cambridge Centre for Analysis. The second author was supported by KU Leuven IF project C14/15/055.}}	
		\author[1]{Georg Maierhofer }
	\affil[1]{Department of Applied Mathematics and Theoretical Physics\\
		University of Cambridge\\ United Kingdom\\g.maierhofer@maths.cam.ac.uk \vspace{1em}}
	\author[2]{Daan Huybrechs}	
	\affil[2]{Department of Computer Science\\
		KU Leuven\\ Belgium\\
		daan.huybrechs@kuleuven.be}	
	
	\maketitle
	
	\begin{abstract}
		Collocation boundary element methods for integral equations are easier to implement than Galerkin methods because the elements of the discretization matrix are given by lower-dimensional integrals. For that same reason, the matrix assembly also requires fewer computations. However, collocation methods typically yield slower convergence rates and less robustness, compared to Galerkin methods. We explore the extent to which oversampled collocation can improve both robustness and convergence rates. We show that in some cases convergence rates can actually be higher than the corresponding Galerkin method, although this requires oversampling at a faster than linear rate. In most cases of practical interest, oversampling at least lowers the error by a constant factor. This can still be a substantial improvement: we analyze an example where linear oversampling by a constant factor $J$ (leading to a rectangular system of size $JN \times N$) improves the error at a cubic rate in the constant $J$. Furthermore, the oversampled collocation method is much less affected by a poor choice of collocation points, as we show how oversampling can lead to guaranteed convergence. Numerical experiments are included for the two-dimensional Helmholtz equation.
	\end{abstract}

{\small\paragraph{Keywords} Fredholm Integral Equations {\Large$\cdot$} Collocation Methods {\Large$\cdot$} Least-squares oversampling {\Large$\cdot$}\newline\begin{flushright} \vspace{-0.6cm}Convergence Analysis\ \ \ \ \ \ \ \ \ \ \ \ \ \ \ \ \ \ \ \ \ \ \ \ \ \ \ \ \ \ \ \ \ \ \ \ \ \ \ \ \ \ \ \ \ \ \ \ \ \ \ \ \ \ \ \ \ \ \ \ \ \ \ \ \ \ \ \ \ \ \ \ \ \ \ \ \ \ \ \ \ \ \ \ \ \ \ \ \ \ 
	\end{flushright}
}

\newpage
\section{Introduction}

Over the recent decade, the concept of oversampling (i.e. taking more observations than the dimension of the trial space) has found increasing attention in numerical analysis as a method to achieve enhanced reconstruction and function approximation. In a range of settings, it is now understood that the effects of suboptimal observations can be mitigated and reconstructions stabilised by introducing a sufficient number of additional observations. As one of the first works in this direction we mention Adcock and Hansen \cite{adcock2012357} who found that oversampling provides a suitable paradigm for function approximation by sampling from a Riesz basis in a Hilbert space, even when the sampling and trial spaces are distinct. It was then shown by Adcock et al. \cite{adcock2014} that oversampling can be used for equispaced Fourier extensions to achieve superalgebraic convergence in a numerically stable manner (where it is known that no method for the Fourier extension problem can be both numerically stable and exponentially convergent). More recently oversampling was studied in the context of function approximation using frames by Adcock and Huybrechs \cite{FNA2} who found it can lead to much-improved accuracy in the approximation and help further mitigate ill-conditioning arising from using a redundant set rather than a basis if an appropriate regularisation is used.

In the present work, we are interested in the study of least-squares oversampling in collocation methods for Fredholm integral equations. These integral equations are of particular practical interest due to the boundary integral method which transforms linear partial differential equations on a domain to Fredholm integral equations on the boundary (see Colton \& Kress \cite{colton2013integral} and Hackbusch \cite{hackbusch2012integral}). The use of collocation to solve these types of integral equations is generally speaking a delicate matter. On the one hand, the discretization matrix entries are given by lower-dimensional integrals, which makes the methods easier to implement and permits the use of a wider range of techniques for numerical integration (see for instance Dea\~no et al. \cite{deano2017computing}, Gibbs et al. \cite{gibbs2020} and Maierhofer et al. \cite{maierhofer2020extended}). On the other hand, at present no general framework for the convergence analysis of collocation methods exists. This is in stark contrast to Galerkin methods, for which there is a well-known and wide-ranging convergence theory based mainly on the coercivity of associated bilinear forms. Even though a general framework for the convergence analysis of collocation methods is not available, the literature offers a number of deep insights into convergence properties in specific settings. One of the most complete studies by Arnold \& Wendland \cite{arnold1983asymptotic} provides guarantees for optimal convergence rates for integral equations on 2D smooth Jordan curves for odd degree spline approximations. Their work is based on a coercivity assumption of the integral operator with respect to the inner product on Sobolev spaces $H^j(\Gamma),j\geq 1$, and shows that the corresponding collocation methods are convergent, albeit at a slower rate than the associated Galerkin methods. These estimates were further extended to even degree splines on Jordan curves (subject to a pseudo-differential form of the integral operator) by Saranen \& Wendland \cite{saranen1985asymptotic}. A unified analysis of spline collocation for strongly elliptic boundary integral and integro-differential operators is given by Arnold \& Wendland \cite{Arnold1985}. Recently there have also been some advances in the analysis of collocation methods for integral equations on higher dimensional spaces, although these have been restricted to biperiodic spaces where the use of Fourier series is available, such as the work by Arens \& R\"osch \cite{arens2016} for integral operators with weakly singular kernels whose singularity can be removed by a transformation to polar coordinates.

This downside of the decreased convergence rate of collocation methods has attracted further research by Sloan \cite{sloan1988quadrature}, Sloan \& Wendland \cite{sloan1989quadrature} and Chandler \& Wendland \cite{chandler1990}, who developed the so-called qualocation method (or `quadrature-modified collocation method') which essentially expresses the inner product in the Galerkin equations through a discrete quadrature rule that is specifically chosen to optimize the convergence rate of the overall method. While the results are very promising - already a linear amount of oversampling in the quadrature points with appropriate weights leads to superconvergence at the same rate as the Galerkin method - their results are highly specific to the setting of equispaced sampling and spline spaces on smooth domains in 2D.

A further level of discretisation from qualocation methods leads to fully discrete methods, for instance Nystr\"om methods as described by Bremer \& Gimbutas \cite{bremer2012} and by Hao et al. \cite{hao2014high}, in which a discrete representation of the integral operator is typically chosen with great care and typically tailored quadrature schemes are incorporated to ensure good convergence properties of the overall method. This means often the discretisation needs to be adapted to the type of integral equation (and specifically the singular kernel of the integral operator) at hand. Similar to collocation methods there is no unified framework for the analysis of Nystr\"om methods and rigorous convergence studies exist only for isolated cases, see for instance the work of Bruno et al. \cite{bruno2013convergence}.

In the present work, we raise the question of how one may improve the convergence properties of collocation methods by introducing oversampling without having to choose the collocation points optimally. Intuitively one might expect that, if used appropriately, any information obtained from additional collocation points can help improve the quality of the approximate solution. One should expect that this is the case even when sampling and basis are not perfectly matched, or when collocation points are chosen suboptimally. This observation was recently verified for a number of practical settings involving wave scattering problems (see for instance Barnett \& Betcke \cite{barnett2008stability}, Gibbs et al. \cite{gibbs2020} and Huybrechs \& Olteanu \cite{HUYBRECHS201992}). In the current manuscript, we aim to provide an introductory but rigorous analysis of the least-squares oversampled collocation method. Although the results focus on some specific cases for 2D integral equations, we believe the analysis highlights important reasons why and how oversampling works (since it provides an approximation to a Bubnov--Galerkin method) and what the determining factor for optimal rates of oversampling in the integral equation setting is (the quality of approximation in the corresponding trapezoidal rule in the relevant function spaces).

\subsection{Main results and outline of the paper}

\begin{figure}[ht]
\begin{subfigure}{0.65\textwidth}
	\centering
	\includegraphics[width=0.99\linewidth]{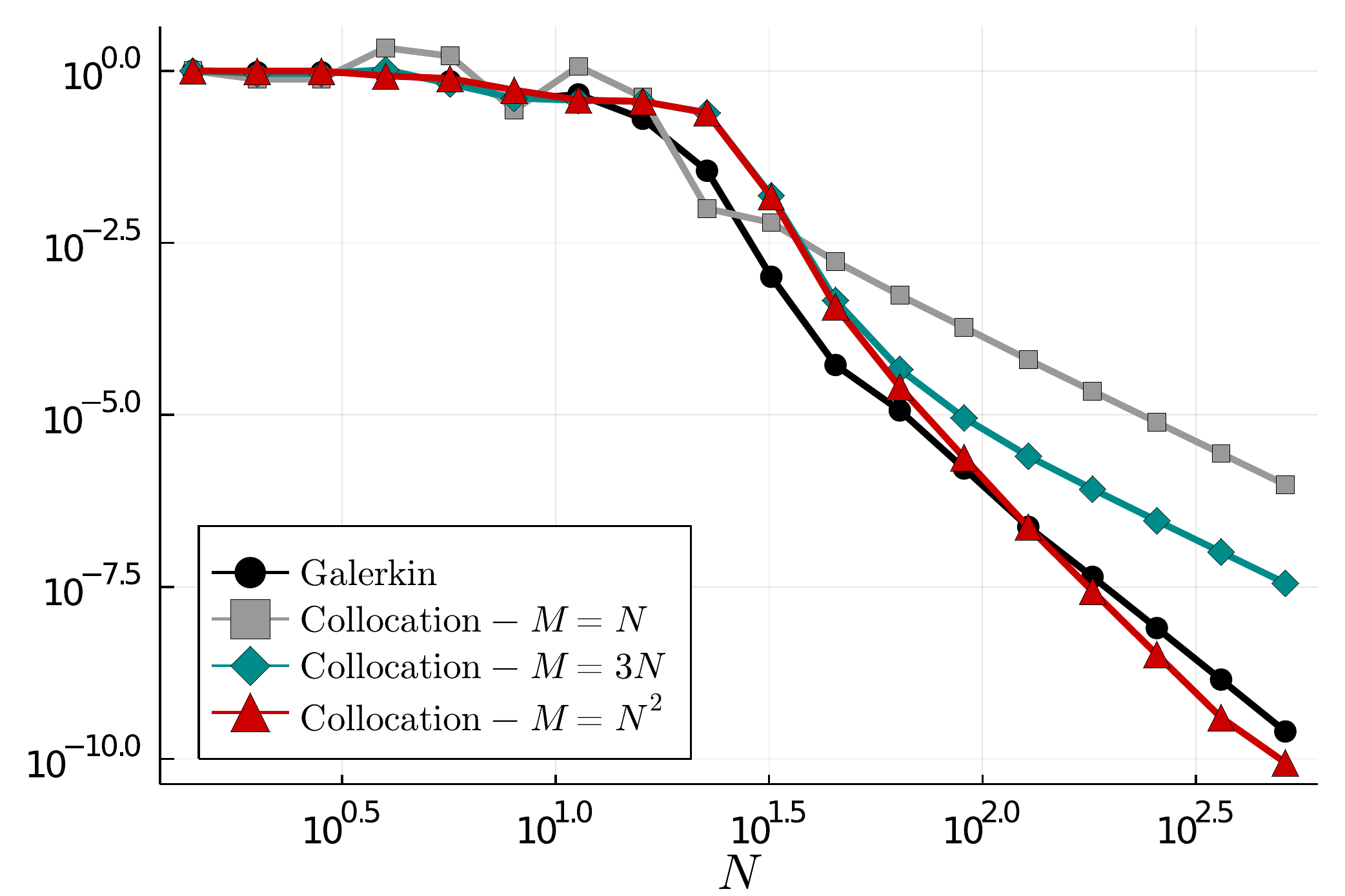}
	\caption{Convergence rates in the field point.}
	\end{subfigure}%
	\begin{subfigure}{0.35\textwidth}
	\centering
	\vspace{0.5cm}
	\includegraphics[width=0.99\linewidth]{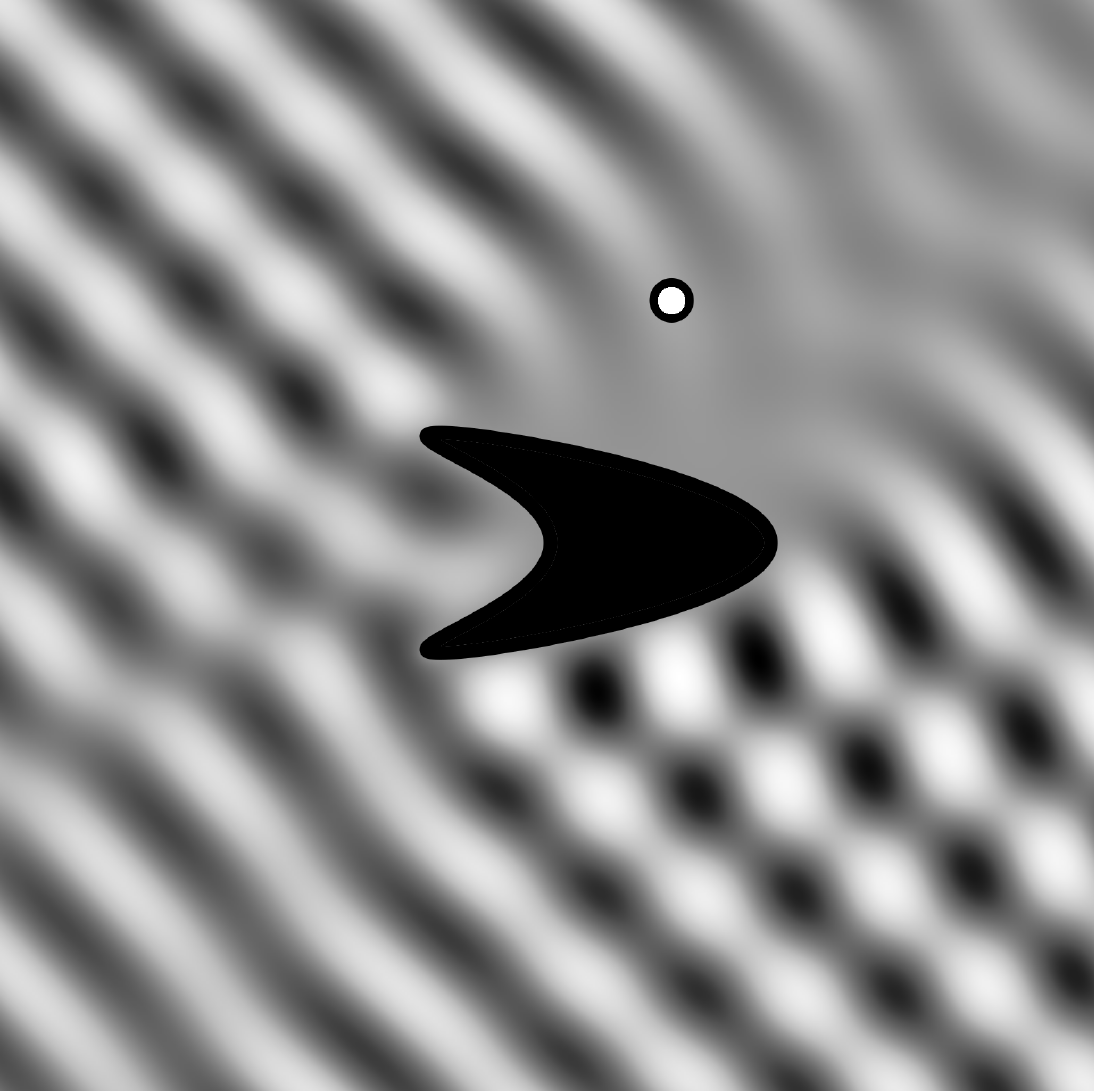}
	\vspace{0.225cm}
	\caption{The geometry and field point.}
	\label{fig:kite_geometry}
	\end{subfigure}
	\caption{The improved convergence properties of oversampled collocation on a smooth scatterer (error in a field point as shown in the right panel). The error for collocation with linear oversampling ($M=3N$) initially follows the Galerkin rate, and eventually follows the standard asymptotic collocation rate ($M=N$) with a  smaller constant. We will show that faster-than-linear oversampling may result in higher asymptotic convergence rates, slightly higher even than Galerkin in this example ($M=N^2$).}
	\label{fig:intro_example_singlelayeroversamplingfieldpointkite}
\end{figure}

The main theoretical results of the paper are Theorems \ref{thm:convergence_in_energy_space}, \ref{thm:discrete_aubin_nitsche} and \ref{thm:convergence_rates_equispaced_grids}. They are briefly illustrated with the numerical results in Fig.~\ref{fig:intro_example_singlelayeroversamplingfieldpointkite}, and more elaborately in \S\ref{sec:numerical_results}.

We introduce the general formulation of the oversampled collocation method together with necessary mathematical assumptions for later analysis in \S\ref{sec:oversapled_collocation_method}. This is followed by a rigorous convergence analysis of the method in \S\ref{sec:convergence_analysis_of_oversampled_collocation}. In Theorem~\ref{thm:convergence_in_energy_space} we prove convergence of the oversampled collocation method in the energy space for a broad class of boundary integral operators on smooth domains, subject to using certain regular boundary element spaces that include commonly used basis functions such as piecewise polynomials (Definition~\ref{def:approximation_property} and Definition~\ref{def:inverse_property}). The convergence is related to a quadrature error estimate in a particular form \eqref{eqn:discrete_inner_product_equispaced_points}, which can be defined and analyzed for arbitrary sets of collocation points. Here, the main point is that the quadrature error may not be small when $M=N$, but it can be made arbitrarily small by increasing $M$ relative to $N$ with minimal assumptions on the points. Though the dimension can be arbitrary, our examples are limited to integral equations on $2D$ domains. We illustrate the robustness in \S\ref{sec:numerical_results} by choosing random points.

In Theorem~\ref{thm:discrete_aubin_nitsche} we analyze convergence rates and show that the density function computed by an oversampled collocation method may converge at optimal rates in a range of Sobolev spaces, depending on the amount of oversampling. This result represents an extension of the so-called Aubin--Nitsche lemma and shows that the lowest possible order of the Sobolev spaces in this range dictates the highest order of convergence of the corresponding error at a field point. This demonstrates superconvergence properties of the oversampled collocation method, though only for domains with a smooth boundary and with more than linear oversampling. The latter is, of course, computationally less desirable as it increases the computational complexity of the solver.

Finally, in Theorem~\ref{thm:convergence_rates_equispaced_grids}, based on earlier results by Sloan, Chandler and Wendland (\cite{chandler1990,sloan1988quadrature,sloan1989quadrature}) we quantify the impact of the oversampling factor in the computationally more favourable regime of linear oversampling, where $M = JN$ for constant $J$. The analysis is restricted to integral operators of a specific pseudo-differential form on smooth domains in 2D, using spline basis functions and matching (oversampled) equispaced collocation points. This setting is the most restrictive, but also the most explicit and shows that the error may decay quite rapidly with $J$. Though the asymptotic convergence rate in the linear oversampling regime remains unchanged compared to standard collocation, the constant involved decays rapidly like $J^{-3}$ when using linear splines for a first kind integral equation using the Laplace single layer potential on a circular domain.

This analysis is followed by a discussion of the results in more specific settings in \S\ref{sec:specific_settings}, including non-equispaced collocation points and Lipschitz domains. These and the aforementioned theoretical results are exhibited on a number of numerical examples in \S\ref{sec:numerical_results} and we conclude the manuscript with a summary of the main insights in this work and an outlook to future research in \S\ref{sec:conclusions}.

\section{An oversampled collocation method}
\label{sec:oversapled_collocation_method}

We first describe the overall set-up and notation before embarking on a more detailed analysis in \S\ref{sec:convergence_analysis_of_oversampled_collocation}. We are given an approximation (trial) space $S_N$, with $\dim S_N=N$, and a domain $\Gamma$ which we assume to be a plane Jordan curve with a regular parametrisation, i.e. it is the graph of  a continuous 1-periodic function 
\begin{align}
\begin{split}\label{eqn:parametrisation_of_boundary}
	z:[0,1)&\rightarrow \mathbb{R}^2\\
	t&\mapsto (x(t),y(t)).
\end{split}
\end{align}
The two main cases of concern in this paper are when $\Gamma$ is smooth (i.e. when $z$ is a diffeomorphism) and when $\Gamma$ is a polygon (i.e. $z$ is piecewise linear). We consider an integral equation on $\Gamma$ of the form 
\begin{align}\label{eqn:original_integral_equation}
	Vu=f,
\end{align}
where $V:\mathcal{H}_1(\Gamma)\rightarrow \mathcal{H}_2(\Gamma)$ is a continuous linear map between two Hilbert spaces $\mathcal{H}_1(\Gamma),\mathcal{H}_2(\Gamma)$ of functions on the boundary. Our method of solution is a \textit{least-squares oversampled collocation method}, whereby we choose $M\geq N$ distinct collocation points $x_m\in \Gamma, m=1,\dots,M$, and define an approximation $u_N^{(M)}$ to the true solution $\tilde{u}$ in the space $S_N=\mathrm{span}\left\{\phi_n\right\}_{n=1}^N$ as follows. We expand
\begin{align}\label{eqn:basis_function_expansion_least_squares_solution}
	u_N=\sum_{n=1}^N a_n \phi_n,
\end{align}
after which the collocation conditions at $x_m, m=1,\dots,M,$ provide an overdetermined $M\times N$ linear system for $\bm{a}$. Motivated by results in approximation theory \cite{FNA1,FNA2} we consider a weighted least-squares solution to this system, such that
\begin{align}\label{eqn:general_least_squares_system}
	\left(\mathcal{G}_{M,N}^\dagger\mathcal{W}_M\mathcal{G}_{M,N}\right)\bm{a}=\mathcal{G}_{M,N}^\dagger\mathcal{W}_M\bm{f}
\end{align}
where
\begin{align*}
	\bm{f}=\left(f(x_m)\right)_{m=1}^M, \quad \mathcal{G}_{M,N}=\left(V\phi_n(x_m)\right)_{m=1,n=1}^{M,N},\quad \mathcal{W}_M=\mathrm{diag}\left(\frac{1}{2}\int_{x_{j-1}}^{x_{j+1}}ds_\Gamma\right)_{m=1}^M.
\end{align*}
Here $\mathcal{W}_M$ is an $M\times M$ diagonal matrix with entries corresponding to distances between sampling points in arclength along $\Gamma$ and it is understood that $x_{N+1}=x_1$. This method reduces to the \textit{standard collocation method} when $M=N$, and in this work we are specifically interested in the convergence properties, as $N\rightarrow\infty$, of the method when $M=M(N)>N$.

\subsection{Mathematical assumptions}

We focus our attention on integral operators $V$ with the following mapping properties:
\begin{itemize}
\item $V$ is a continuous linear map
\begin{align}\label{eqn:continuity_of_V}
	V:H^{s+\alpha}(\Gamma)\rightarrow H^{s-\alpha}(\Gamma)
\end{align}
for some $\alpha\in\mathbb{R}$ and any $s\in\mathbb{R}$ with $|s|\leq s_0$, some fixed constant $s_0$.
\item The inverse of $V$ is a well-defined continuous linear map
\begin{align}\label{eqn:continuity_of_inverse}
	V^{-1}:H^{s-\alpha}(\Gamma)\rightarrow H^{s+\alpha}(\Gamma)
\end{align}
for any $s\in\mathbb{R}$ with $|s|\leq s_1$.
\end{itemize}
In the usual way we call $2\alpha$ the order of the operator. In the first instance, further on in \S\ref{sec:convergence_analysis_of_oversampled_collocation}, we will focus on the case when $\Gamma$ is smooth, which means through the parametrisation z (see \eqref{eqn:parametrisation_of_boundary}) there is a one-to-one correspondence between functions on $\Gamma$ and 1-periodic functions on $[0,1)$. Moreover, since $z$ is a diffeomorphism the Sobolev norms on $\Gamma$ are equivalent to the Sobolev norms of the mapped function on $[0,1)$, i.e. for any $\alpha\in\mathbb{R}$ there is $C_\alpha$, such that
\begin{align*}
    C_{\alpha}^{-1}\|g\|_{H^{\alpha}(\Gamma)}\leq \|g\circ z\|_{H^{\alpha}([0,1))}\leq C_{\alpha}\|g\|_{H^{\alpha}(\Gamma)}.
\end{align*}
Thus from now on we shall limit ourselves to $\Gamma=[0,1)$, and unless mentioned otherwise it is understood that $H^t=H^t([0,1))$. In particular, on $[0,1)$, we use the following definitions of the Sobolev norm and the $L^2$-duality pairing for $f\in H^{s},g\in H^{-s}$:
\begin{align}\begin{split}\label{eqn:definition_sobolev_norms_in_terms_of_fourier_modes}
   \|f\|_{s}&=\|f\|_{H^s}=\|f\|_{H^s([0,1))}:=\left(|\hat{f}_0|^2+\sum_{0\neq m\in\mathbb{Z}}|m|^{2s}|\hat{f}_m|^2\right)^{1/2},\\
    \langle f,g\rangle&:=\overline{\hat{f}_0}\hat{g}_0+\sum_{0\neq m\in\mathbb{Z}}\overline{\hat{f}_m}\hat{g}_m=\int_0^1\overline{f(x)}g(x)dx,
    \end{split}
\end{align}
where
\begin{align*}
    \hat{f}_m:=\int_0^1e^{-2\pi i m t}f(t)dt.
\end{align*}
Note that in \S\ref{sec:discrete_aubin_nitsche_lemma} we will also make use of the expression
\begin{align*}
    \|f\|_{H^s}=\sup_{\substack{g\in C^\infty\\\|g\|_{H^{-s}}=1}}|\langle f,g\rangle|.
\end{align*}
Finally, in \S\ref{sec:polygonal_domains} we lift some of the above assumptions, but we will make clear at that point which properties still hold.

\begin{example}[Integral formulations of the Helmholtz equation, see \cite{colton2013integral}]\label{ex:Helmholtz_boundary_integral_formulation}
If we solve the exterior problem for the Helmholtz equation on some domain $\Omega$, with $\Gamma=\partial \Omega$, it is well-known that the Dirichlet problem
\begin{align*}
    \begin{cases}
    \Delta \phi +k^2 \phi=0,&\quad \mathbb{R}^2\setminus\overline{\Omega}\\
    \phi(x)=g(x),&\quad x\in \Gamma,
    \end{cases}
\end{align*}
can be formulated in terms of the single or double layer potential formulations:
\begin{align*}
\mathcal{S}u=f,\quad \left(\frac{1}{2}\mathcal{I}+\mathcal{D}\right)u=f,    
\end{align*}
where the single layer potential $\mathcal{S}$ and double layer potential operators $\mathcal{D}$ are given by
\begin{align*}
    (\mathcal{S}\phi)(x)&=\int_\Gamma G(x,y)\phi(y) ds_y,\quad x\in\Gamma,\\
    (\mathcal{D}\phi)(x)&=\int_\Gamma\frac{\partial G}{\partial n_y}(x,y)\phi(y)ds_y,\quad x\in\Gamma,
\end{align*}
where $G(x,y)=i/4H^{(1)}_0(k|x-y|)$ is the Green's function of the Helmholtz equation with wavenumber $k$. These integral operators have extensions to $H^r$ for any $r\in \mathbb{R}$ which satisfy, away from resonant frequencies, the above assumptions \eqref{eqn:continuity_of_V}\&\eqref{eqn:continuity_of_inverse} for any $s\in\mathbb{R}$ with orders of $2\alpha=-1$ and $2\alpha=0$ respectively.
\end{example}
We consider $(l,m)$-regular boundary element spaces $S_N=S_{h}^{l,m}\subset H^{m+1/2}$ satisfying the inverse assumption in the sense of Babu\v{s}ka \& Aziz \cite[Section 4.1]{babuska1992} (cf. also \cite[p. 38]{Hsiao2017}). These spaces are defined through two important properties:

\begin{definition}[Approximation property]\label{def:approximation_property} Let $t\leq s\leq l$ and $t<m+\frac{1}{2}$. Assume there exists a constant $c$ such that for any $v\in H^{s}(\Gamma)$, a sequence $\chi_h\in S_{h}^{l,m}$ exists and satisfies the estimate
	\begin{align}\label{eqn:approximation_property}
		\|v-\chi_h\|_{H^{t}(\Gamma)}\leq c h^{s-t}\|v\|_{H^{s}(\Gamma)}.
		\end{align}
	The space $S_h^{l,m}$ is called $(l,m)$-regular if for any fixed $\sigma<m+1/2$ there is such a sequence $\chi_h$ which satisfies \eqref{eqn:approximation_property} for all $t>\sigma$, i.e. which can be chosen independently of $t$.
\end{definition}

\begin{definition}[Inverse property]\label{def:inverse_property} For $t\leq s\leq m+\frac{1}{2}$, there exists a constant $C$ such that for all $\chi_h\in S_{h}^{l,m}$,
	\begin{align*}
		\|\chi_N\|_{H^{s}(\Gamma)}\leq C h^{t-s}\|\chi_N\|_{H^{t}(\Gamma)}.
	\end{align*}
\end{definition}
Note that these definitions introduce a parameter $h$, which typically tends to zero in a sequence of approximations. This is the parameter $h$ used in the convergence statements and in the conditions of the theorems further on in \S\ref{sec:convergence_analysis_of_oversampled_collocation}.

\begin{example}[Spline spaces on $\rho$-quasiuniform mesh, cf. {\cite[p. 359]{arnold1983asymptotic}}]\label{ex:splines_as_regular_boundary_element_spaces} A mesh
\begin{align*}
    \Delta_N=\left\{0=x_1<x_2<\cdots<x_N<1\right\}
\end{align*}
is called $\rho$-quasiuniform ($\rho>0$) if
	\begin{align*}
		\max_{1\leq j\leq N}|x_{j+1}-x_{j}|\leq \rho \min_{1\leq j\leq N}|x_{j+1}-x_{j}|
	\end{align*}
where it is understood that $x_{N+1}=x_{1}$. Splines of degree $d$ on a sequence of quasiuniform meshes $\Delta_N$, where $h_N=\max_{1\leq j\leq N}|x_{j+1}-x_j|$ are regular boundary element spaces on smooth periodic curves in the above sense, with $S_N=S(\Delta_N)=S_{h_N}^{d+1,d}$.
\end{example}

\subsection{From least-squares to a discrete Bubnov-Galerkin method}\label{sec:least-squares_to_bubnov-galerkin}

Oversampled collocation leads to a rectangular linear system and this system is solved in a least-squares sense, recall~\eqref{eqn:general_least_squares_system}. In view of the chosen weights, we will see that these normal equations have a continuous limit in the regime where $M \to \infty$ for fixed $N$. However, that limit differs from the classical Galerkin method of the same integral equation. The latter leads to the orthogonality conditions
\begin{equation}\label{eqn:classical_galerkin_orthogonality_conditions}
        \langle \chi_N,Vu_N\rangle=\langle\chi_N,f\rangle,\quad \forall\chi_N\in S_N.
\end{equation}

The central observation in the following analysis is that the least-squares system \eqref{eqn:general_least_squares_system} amounts to a discrete Bubnov-Galerkin method instead, in the following sense. Let $\Delta_M=\{0\leq x_1<x_2<\cdots<x_M<1\}$ be the collocation points and let $u_N^{(M)}\in S_N$ be the least-squares collocation approximation to the true solution $\tilde{u}$ in the sense of \eqref{eqn:basis_function_expansion_least_squares_solution}--\eqref{eqn:general_least_squares_system}, then we note that \eqref{eqn:general_least_squares_system} is equivalent to
\begin{align}\label{eqn:discrete_orthogonality_condition}
    \left\langle V\chi_N,Vu_N^{(M)}\right\rangle_{M}=\left\langle V\chi_N,V\tilde{u}\right\rangle_{M},\quad \forall \chi_N\in S_N,
\end{align}
where we defined the discrete inner product to be
\begin{align*}
    \left\langle f,g\right\rangle_M=\sum_{m=1}^M\frac{|x_{j+1}-x_{j-1}|}{2}\overline{f(x_j)}g(x_j),
\end{align*}
where it is understood that $|x_{N+1}-x_{N-1}|=x_{1}+1-x_{N-1}$, and we used the fact that on $\Gamma=[0,1)$ it is the case that  $\int_{x_{j-1}}^{x_{j+1}}ds_\Gamma=x_{j+1}-x_{j-1}$. The \textit{discrete orthogonality condition} \eqref{eqn:discrete_orthogonality_condition} plays a central role in the analysis of \S\ref{sec:convergence_analysis_of_oversampled_collocation}.

\begin{remark}\label{rem:large_oversampling_limit}
Note that, compared to the classical Galerkin orthogonality conditions~\eqref{eqn:classical_galerkin_orthogonality_conditions}, the discrete conditions \eqref{eqn:discrete_orthogonality_condition} feature an additional integral operator $V$ in the first argument of the inner product. This is the case both in the left hand side and in the right hand side: the integral equation is projected using the basis $\{V \chi_N\}_{\chi_N \in S_N}$ rather than $\{ \chi_N\}_{\chi_N \in S_N}$. If $V$ has negative order, then the former is smoother than the latter and that underlies some of the differences in convergence rates between Galerkin and the large oversampling limit of the collocation method.
\end{remark}

\begin{remark}
One might also wish to consider a discretisation of the standard Galerkin method \eqref{eqn:classical_galerkin_orthogonality_conditions}, which would result in the discrete orthogonality conditions
\begin{align}\label{eqn:modified_discrete_orthogonality_condition}
 \left\langle \chi_N,Vu_N^{(M)}\right\rangle_{M}=\left\langle \chi_N,V\tilde{u}\right\rangle_{M},\quad \forall \chi_N\in S_N.
\end{align}
We call this the modified oversampled collocation method because instead of weighted normal equations \eqref{eqn:general_least_squares_system} it can be formulated as
\begin{align*}
	\left(\mathcal{B}_{M,N}^\dagger\mathcal{W}_M\mathcal{G}_{M,N}\right)\bm{a}=\mathcal{B}_{M,N}^\dagger\mathcal{W}_M\bm{f}
\end{align*}
where
\begin{align*}
	\bm{f}&=\left(f(x_m)\right)_{m=1}^M, \quad \mathcal{G}_{M,N}=\left(V\phi_n(x_m)\right)_{m=1,n=1}^{M,N},\\
	\mathcal{B}_{M,N}&=\left(\phi_n(x_m)\right)_{m=1,n=1}^{M,N},\quad \mathcal{W}_M=\mathrm{diag}\left(\frac{1}{2}\int_{x_{j-1}}^{x_{j+1}}ds_\Gamma\right)_{m=1}^M.
\end{align*}
For this method, most of the analysis in the following sections carries through in a similar fashion, however some of the assumptions on $V$ such as invertibility are no longer sufficient to guarantee uniform ellipticity of the corresponding discrete forms and it is necessary to impose a discrete inf-sup assumption on the operator $V$ (similar to what would be required for the continuous Galerkin method \eqref{eqn:classical_galerkin_orthogonality_conditions}).
\end{remark}

\section{Convergence analysis of the oversampled collocation method}\label{sec:convergence_analysis_of_oversampled_collocation}
As mentioned above the central idea in the analysis is to regard the oversampled collocation method as a discrete version of a  Bubnov–Galerkin method, by which we mean the approximation $u_N\in S_N$ which is defined through the continuous orthogonality conditions
\begin{align}\label{eqn:bubnov_galerkin_method}
    \left\langle V\chi_N,Vu_N\right\rangle_{L^2}=\left\langle V\chi_N,f\right\rangle_{L^2},\quad \forall \chi_N\in S_N.
\end{align}
A combination of Strang-type estimates (cf. Theorem~4.1.1 and Theorem~4.2.2 in \cite{ciarlet2002finite}) and the error of the trapezoidal rule for the inner product allows us to study the convergence rates of our methods. For simplicity we begin by focusing on the case of smooth domains, and assume that
\begin{align}\label{eqn:mapping_properties_V}
	V: H^{s+\alpha}\rightarrow H^{s-\alpha}
\end{align}
is a continuous isomorphism for all $s\in\mathbb{R}$. If the integral equation arises as a boundary integral formulation of the boundary value problem for some partial differential equation, $\mathcal{L}\phi=f$, on a domain $\Omega\subset\mathbb{R}^2$ with $\Gamma=\partial \Omega$, we already saw a useful way of assessing the error in Fig.~\ref{fig:intro_example_singlelayeroversamplingfieldpointkite}. We can judge the error/convergence of the boundary integral method via the evaluation of the field at a point away from the boundary. In the case of Dirichlet or Neumann boundary value problems the field $\phi:\Omega\rightarrow \mathbb{C}$ is typically expressed in the form
\begin{align*}
    \phi(x)=\int_{\Gamma}k(x,y)u(y)ds_y,\quad x\in\Omega
\end{align*}
where $k(x,y)$ is a kernel function related to the Green's function of the partial differential operator $\mathcal{L}$. When $\mathcal{L}$ is elliptic $k(x,y)$ is a smooth function away from $x=y$, which means we can estimate, for any $t\in\mathbb{R}$,
\begin{align}\nonumber
    \left|\phi_N(x)-\phi(x)\right|&=\left|\int_{\Gamma}k(x,y)\left(u_N(y)-\tilde{u}(y)\right)ds_y\right|=\left|\left\langle k(x,\cdot),u_N-\tilde{u} \right\rangle_{L^2(\Gamma)}\right|\\\label{eqn:estimate_field_point_error_in_terms_of_sobolev_norms}
    &\leq \|k(x,\cdot)\|_{H^{t}(\Gamma)}\|\tilde{u}-u_N\|_{H^{-t}(\Gamma)},\quad x\in\Omega.
\end{align}
Since $y\mapsto k(x,y)\in C^\infty(\Gamma)$ for all $x\in \Omega$ the convergence rate of the approximation $\phi_N(x)$ to $\phi(x)$ is governed by the fastest convergence of $\|\tilde{u}-u_N\|_{H^{-t}(\Gamma)}$ in any boundary Sobolev norm for $t\in\mathbb{R}$. We shall see in the coming two sections that optimal convergence rates can be obtained in low order Sobolev spaces for sufficient amounts of oversampling. It will become apparent in \S\ref{sec:polygonal_domains} that the smoothness of the domain boundary affects the range of Sobolev spaces one can consider.

\subsection{Strang estimate for convergence on energy space}
 The first step in the analysis is to show that already a small amount of oversampling is sufficient to guarantee convergence of the \textit{least-squares oversampled collocation method} on the energy space $H^{2\alpha}$. As before, we let $\Delta_M=\{0\leq x_1<\dots<x_M<1\}$ be the collocation points. As the dimension $N$ of the approximation space $S_N=S_{h}^{l,m}$ increases, we also increase the number of collocation points $M=M(N)$ in a predefined way and we assume that there is an error estimate for the proximity of the discrete inner product to a continuous $L^2$ pairing in the following form for some $r,s>0$:
\begin{align}\label{eqn:general_error_estimate_trapezoidal_rule}
    |\langle g,f\rangle-\langle g,f\rangle_M|\leq \mathcal{E}_{r,s}(\Delta_M)\left(\|f\|_r\|g\|_s+\|f\|_s\|g\|_r\right)
\end{align}
for any $f,g\in H^{\max\{s,r\}}$, where $\mathcal{E}_{r,s}(\Delta_M)>0$ is independent of $f,g$. We will derive such error estimates in a number of settings further on.

\begin{proposition}\label{prop:uniform_ellipticity}
If \eqref{eqn:general_error_estimate_trapezoidal_rule} holds for $s,r$ with $\max\{r,s\}<m+1/2-2\alpha$, and if $\Delta_M=\Delta_{M(N)}$ is chosen such that
\begin{align*}
    \lim_{N\rightarrow\infty}\mathcal{E}_{r,s}(\Delta_M)h^{-(r+s)}=0,
\end{align*}
where $h=h(N)$ then the bilinear forms
\begin{align*}
    (f,g)\mapsto \left\langle Vf,Vg\right\rangle_M
\end{align*}
are uniformly $S_N$-elliptic, meaning they satisfy
\begin{align*}
    \gamma\|\chi_N\|_{H^{2\alpha}}^2\leq \left\langle V\chi_N,V\chi_N\right\rangle_M,\quad \forall\chi_N\in S_N,
\end{align*}
for some constant $\gamma$ independent of $N$.
\end{proposition}
\begin{proof}
    By \eqref{eqn:mapping_properties_V} we know that for any $\chi_N\in S_N$
    \begin{align*}
        \|\chi_N\|_{H^{2\alpha}}%&\leq \|V^{-1}\|_{2\alpha\rightarrow 0}^2\left|\left\langle V\chi_N,V \chi_N\right\rangle\right|\\
        &\leq \|V^{-1}\|_{2\alpha\rightarrow 0}^2\left(\left|\left\langle V\chi_N,V \chi_N\right\rangle_M\right|+\left|\left\langle V\chi_N,V \chi_N\right\rangle-\left\langle V\chi_N,V \chi_N\right\rangle_M\right|\right)\\
        &\leq \|V^{-1}\|_{2\alpha\rightarrow 0}^2\left(\left|\left\langle V\chi_N,V \chi_N\right\rangle_M\right|\right.\\
        &\left.\quad+2\mathcal{E}_{r,s}(\Delta_M)\|V\|_{r+2\alpha\rightarrow r}\|V\|_{s+2\alpha\rightarrow s}\|\chi_N\|_{r+2\alpha}\|\chi_N\|_{s+2\alpha}\right).
    \end{align*}
    Thus, by the inverse property of $S_N=S_h^{l,m}$ (Definition \ref{def:inverse_property}) we have, whenever $\max\{r,s\}<m+1/2-2\alpha$,
    \begin{align*}
        \|\chi_N\|_{r+2\alpha}\|\chi_N\|_{s+2\alpha}\leq Ch^{-(r+s)}\|\chi_N\|_{2\alpha}^2
    \end{align*}
    and therefore
    \begin{align*}
         (1-\tilde{C}\mathcal{E}_{r,s}(\Delta_M)h^{-(r+s)})\|\chi_N\|_{H^{2\alpha}}&\leq\|V^{-1}\|_{2\alpha\rightarrow 0}^2\left|\left\langle V\chi_N,V \chi_N\right\rangle_M\right|,
    \end{align*}
so if $\lim_{N\rightarrow\infty}\mathcal{E}_{r,s}(\Delta_M)h^{-(r+s)}=0$ the result follows.
\end{proof}
\begin{theorem}[Strang-type bound for convergence in $H^{2\alpha}$]\label{thm:convergence_in_energy_space}
	If \eqref{eqn:general_error_estimate_trapezoidal_rule} holds for $s,r$ with $\max\{r,s\}<m+1/2-2\alpha$, and $\Delta_M=\Delta_{M(N)}$ is chosen such that
\begin{align*}
    \lim_{N\rightarrow\infty}\mathcal{E}_{r,s}(\Delta_M)h^{-(r+s)}=0,
\end{align*}
then there is a constant $C>0$ independent of $M,N,u$ such that 
	\begin{align*}
		\|u_N^{(M)}-u\|_{H^{2\alpha}}\leq C h^{l-2\alpha}\|u\|_{H^{l}}.
	\end{align*}
	This means the optimal convergence rate in $H^{2\alpha}$ is achieved.
\end{theorem}
\begin{proof}We begin with the following estimate on the energy space $H^{2\alpha}$:
\begin{align*}
\|u_{N}^{(M)}-u\|_{H^{2\alpha}}\leq \inf_{\chi_N\in  S_N}\left\{\|u-\chi_N\|_{H^{2\alpha}}+\|u_{N}^{(M)}-\chi_N\|_{H^{2\alpha}}\right\}.
\end{align*}
    We can estimate the second term using Proposition~\ref{prop:uniform_ellipticity} and the discrete orthogonality condition \eqref{eqn:discrete_orthogonality_condition} as follows
    \begin{align}\nonumber
    \hspace{-0.4cm}\gamma \|u_N^{(M)} -  \chi_N\|_{H^{2\alpha}}^2&\leq \left|\left\langle V(u_N^{(M)}-\chi_N),V(u_N^{(M)}-\chi_N)\right\rangle_M\right|\\\nonumber
    &=\left|\left\langle V(u-\chi_N),V(u_N^{(M)}-\chi_N)\right\rangle_M\right|\\\nonumber
    &\leq \left|\left\langle V(u-\chi_N),V(u_N^{(M)}-\chi_N)\right\rangle\right|\\\nonumber
    &\quad+\left|\left\langle V(u-\chi_N),V(u_N^{(M)}-\chi_N)\right\rangle-\left\langle V(u-\chi_N),V(u_N^{(M)}-\chi_N)\right\rangle_M\right|\\
    \begin{split}\label{eqn:upper_bound_energy_space_estimate}
    &\leq\|V\|_{2\alpha\rightarrow 0}^2\|u-\chi_N\|_{2\alpha}\|u_N^{(M)}-\chi_N\|_{2\alpha}\\
    &\quad+\mathcal{E}_{r,s}(\Delta_M)\|V\|_{r+2\alpha\rightarrow r}\|V\|_{s+2\alpha\rightarrow s}\|u-\chi_N\|_{r+2\alpha}\|\|u_N^{(M)}-\chi_N\|_{s+2\alpha}\\
    &\quad\quad+\mathcal{E}_{r,s}(\Delta_M)\|V\|_{r+2\alpha\rightarrow r}\|V\|_{s+2\alpha\rightarrow s}\|u-\chi_N\|_{s+2\alpha}\|\|u_N^{(M)}-\chi_N\|_{r+2\alpha}
    \end{split}
    \end{align}
    Using the inverse property of $S_N=S_h^{l,m}$ (Definition \ref{def:inverse_property}) we find if $\max\{r,s\}<M+1/2-2\alpha$
\begin{align*}
    \tilde{\gamma}\|u_N^{(M)}-\chi_N\|_{H^{2\alpha}}\leq & \|u-\chi_N\|_{2\alpha}+\mathcal{E}_{r,s}(\Delta_M)h^{-s}\|u-\chi_N\|_{r+2\alpha}+ \mathcal{E}_{r,s}(\Delta_M)h^{-r}\|u-\chi_N\|_{s+2\alpha}
\end{align*}
for some $\tilde{\gamma}>0$ independent of $N,\Delta_M,u$. And thus if
\begin{align*}
    \lim_{N\rightarrow\infty}\mathcal{E}_{r,s}(\Delta_M)h^{-(r+s)}=0,
\end{align*}
we find
\begin{align*}
    \tilde{\gamma}\|u_N^{(M)}-\chi_N\|_{H^{2\alpha}}&\leq\|u-\chi_N\|_{2\alpha}+h^{r}\|u-\chi_N\|_{r+2\alpha}+h^{s}\|u-\chi_N\|_{s+2\alpha}.
\end{align*}
By the approximation property of $S_N=S_h^{l,m}$ (Definition \ref{def:approximation_property}) the result follows.
\end{proof}
We note that Theorem~\ref{thm:convergence_in_energy_space} actually implies optimal convergence in Sobolev norms that are of higher order than the energy space by the following argument.
\begin{corollary}\label{cor:higher_order_sobolev_estimates}
Let $N,\Delta_M=\Delta_{M(N)}$ satisfy the assumptions of Theorem~\ref{thm:convergence_in_energy_space}, then for all $2\alpha<t<m+1/2$ we have optimal convergence in the sense that there is a constant independent of $N,\Delta_M,u$, such that
\begin{align*}
    \|u_N^{(M)}-u\|_{H^{t}}\leq Ch^{l-t}\|u\|_{H^l}.
\end{align*}
\end{corollary}
\begin{proof}
Let $2\alpha<t<m+1/2$, then
    \begin{align*}
        \|u_N^{(M)}-u\|_{H^{t}}&\leq \|u-\psi_N\|_{H^{t}}+\|\psi_N-u_N^{(M)}\|_{H^{t}}\\
        &\leq\|u-\psi_N\|_{H^{t}}+C h^{2\alpha-t}\|\psi_N-u_N^{(M)}\|_{H^{2\alpha}}\\
        &\leq \|u-\psi_N\|_{H^{t}}+C h^{2\alpha-t}\left(\|\psi_N-u\|_{H^{2\alpha}}+\|u-u_N^{(M)}\|_{H^{2\alpha}}\right)
    \end{align*}
    and the result follows by Theorem~\ref{thm:convergence_in_energy_space} and by the uniform approximation property in Definition~\ref{def:approximation_property}.
\end{proof}
\begin{remark}
One can see that the statements in Theorem~\ref{thm:convergence_in_energy_space} and Proposition~\ref{prop:uniform_ellipticity} are in fact true for much more general settings, including the case of 3D boundary integral equations, i.e. when $\dim \Gamma=2$, as long as appropriate error estimates for the discrete inner product similar to \eqref{eqn:general_error_estimate_trapezoidal_rule} are available. However, for the purpose of this paper we shall remain in the 2D setting.
\end{remark}
Although the above statements are phrased in a general form we can use them to make concrete predictions. We begin by considering equispaced collocation points $|x_{j+1}-x_j|=1/M$, then we have the following error estimate for the $L^2$-inner product as in \eqref{eqn:general_error_estimate_trapezoidal_rule}:

\begin{lemma}[Error in discrete $L^2$ inner product - equispaced sampling]\label{lem:bound_equispaced_trapezoidal_rule_L^2_product}
Let $\Delta_M=\{x_m=\tilde{x}+j/M\}_{m=1}^M$ be a set of equispaced collocation points (where it is understood $x+1\equiv x$) and let
\begin{align*}
    \left\langle f,g\right\rangle =\frac{1}{M}\sum_{m=1}^M\overline{f(x_m)}g(x_m)
\end{align*}
for $f,g\in H^r([0,1))$ for some $r>1/2$. Then there is a constant $C_{r,s}>0$ independent of $f,g$ such that for any $r>s>1/2$:
	\begin{align*}
		\left|\langle f,g\rangle-\langle f,g\rangle_M\right|\leq C_{r,s} M^{-r}\Big(\|f\|_{H^r}\|g\|_{H^s}+\|f\|_{H^s}\|g\|_{H^{r}}\Big).
	\end{align*}
\end{lemma}
\begin{proof}
	For completeness the proof is given in Appendix~\ref{app:proof_error_bound_trapezoidal_rule}.
\end{proof}

If we now choose $S_N$ to be a sequence of spline spaces of degree $d$ on a quasiequispaced mesh (cf. Example~\ref{ex:splines_as_regular_boundary_element_spaces}) we find the following convergence result for the oversampled collocation method. Note in particular that here the collocation points need not match the mesh of the basis functions, i.e. this result reflects the idea that a small amount of oversampling can guarantee convergence even if the collocation points are chosen suboptimally:

\begin{corollary}\label{cor:convergence_on_equispaced_grids_spline_basis}
	If $\Delta_M$ are equispaced and $S_N$ are spline spaces of degree $d$ on a quasiequispaced mesh, and if in addition $M=M(N)\geq N^{\beta}$ for some $\beta>1+\frac{1}{2d+1-4\alpha}$, then there is a constant $C$ independent of $M,N,u$ such that 
	\begin{align*}
		\|u_N^{(M)}-u\|_{H^{2\alpha}}\leq C N^{2\alpha-d-1}\|u\|_{H^{d+1}}.
	\end{align*}
\end{corollary}

\subsection{Superconvergence and the discrete Aubin--Nitsche Lemma}\label{sec:discrete_aubin_nitsche_lemma}
Although the results in Theorem~\ref{thm:convergence_in_energy_space} guarantee convergence of the oversampled collocation method we have yet to ask at what asymptotic rate we expect this to occur. For continuous Galerkin methods it is possible to prove superconvergence by a duality argument, the so-called Aubin--Nitsche lemma (see \cite[\S4.3]{Hsiao2017} and \cite{hsiao1981}). We will demonstrate that a sufficient amount of superlinear oversampling can actually achieve such superconvergence for the oversampled collocation method as well, which explains the results observed in Fig.~\ref{fig:intro_example_singlelayeroversamplingfieldpointkite}.
\begin{theorem}[Discrete Aubin--Nitsche Lemma]\label{thm:discrete_aubin_nitsche}
Let $-l\leq t\leq 0$ and suppose \eqref{eqn:general_error_estimate_trapezoidal_rule} holds for $s,r$ with $\max\{r,s\}<\min\{m+1/2,-t\}-2\alpha$, and that $\Delta_M=\Delta_{M(N)}$ is chosen such that
\begin{align*}
    \lim_{N\rightarrow\infty}\mathcal{E}_{r,s}(\Delta_M)h^{2\alpha-t-\max\{r,s\}}=0.
\end{align*}
Then there is a constant $C>0$ independent of $M,N,u$ such that 
	\begin{align*}
		\|u_N^{(M)}-u\|_{H^{t+4\alpha}}\leq C h^{l-t-4\alpha}\|u\|_{H^{l}}.
	\end{align*}
	This means the optimal convergence rate in $H^{t+4\alpha}$ is achieved.
\end{theorem}
\begin{proof}
Note we have by the definition of the dual norm, for all $t\leq 0$:
\begin{align}\nonumber
	\|(V^*V)(u-u_N^{(M)})\|_{H^t}&=\sup_{\substack{\psi\in C^\infty\\\|\psi\|_{H^{-t}}=1}}|\langle (V^*V)(u-u_N^{(M)}),\psi\rangle| \\\nonumber& =\sup_{\substack{\psi\in C^\infty\\\|\psi\|_{H^{-t}}=1}}|\langle Vu-Vu_N^{(M)}),V\psi\rangle| \\\nonumber
	&\hspace{-0cm}\leq \sup_{\substack{\psi\in C^\infty\\\|\psi\|_{H^{-t}}=1}}\inf_{\chi_N\in  S_N}\left(|\langle Vu-Vu_N^{(M)},V\psi-V\chi_N\rangle|+|\langle Vu-Vu_N^{(M)},V\chi_N\rangle|\right)\\\nonumber
	&\hspace{-0cm}= \sup_{\substack{\psi\in C^\infty\\\|\psi\|_{H^{-t}}=1}}\inf_{\chi_N\in  S_N}\left(|\langle Vu-Vu_N^{(M)},V\psi-V\chi_N\rangle|\right. \\ \label{eqn:collocation_use_of_orthogonal_projection}
	&\hspace{0.75cm}+\left.|\langle Vu-Vu_N^{(M)},V\chi_N\rangle-\langle V(u-u_N^{(M)}),V\chi_N\rangle_M|\right)\\\label{eqn:aubin_nitsche_trapezoidal_rule_included}
	&\hspace{-0cm}\leq \sup_{\substack{\psi\in C^\infty\\\|\psi\|_{H^{-t}}=1}}\inf_{\chi_N\in  S_N}\left(\|V\|^2_{2\alpha\rightarrow0}\|u-u_N^{(M)}\|_{H^{2\alpha}}\|\psi-\chi_N\|_{H^{2\alpha}}\right.\\\nonumber
	&\hspace{0.75cm}+\mathcal{E}_{r,s}(\Delta_M)\|V\|_{r+2\alpha\rightarrow r}\|V\|_{s+2\alpha\rightarrow s}\|u-u_N^{(M)}\|_{r+2\alpha}\|\|\chi_N\|_{s+2\alpha}\\\nonumber
    &\hspace{1.5cm}\left.+\mathcal{E}_{r,s}(\Delta_M)\|V\|_{r+2\alpha\rightarrow r}\|V\|_{s+2\alpha\rightarrow s}\|u-u_N^{(M)}\|_{s+2\alpha}\|\|\chi_N\|_{r+2\alpha}\right).
\end{align}
Here \eqref{eqn:collocation_use_of_orthogonal_projection} follows from \eqref{eqn:discrete_orthogonality_condition} and \eqref{eqn:aubin_nitsche_trapezoidal_rule_included} follows from \eqref{eqn:general_error_estimate_trapezoidal_rule} since $\max\{r,s\}<m+1/2-2\alpha$. We now refer back to the approximation property of the basis spaces $S_N$ (Definition \ref{def:approximation_property}) which shows that for any $\psi\in H^{-t}$ we have $\chi_N\in S_N$ such that (uniformly for all $-l<s<m+1/2$ and $s\leq -t\leq l$):
\begin{align*}
	\|\psi-\chi_N\|_{H^s}<C h^{-(s+t)}\|\psi\|_{H^{-t}}.
\end{align*}
for this choice of $\chi_N$ it also follows that
\begin{align*}
	\|\chi_N\|_{H^s}&\leq \|\psi-\chi_N\|_{H^s}+\|\psi\|_{H^s}\\
	&\leq \|\psi-\chi_N\|_{H^s}+\|\psi\|_{H^{-t}}\leq (1+Ch^{-(s+t)})\|\psi\|_{H^{-t}}\leq \tilde{C}\|\psi\|_{H^{-t}},
\end{align*}
for some $\tilde{C}>0$ independent of $\psi$. Thus, choosing $\chi_N$ in this way, the right hand side of \eqref{eqn:aubin_nitsche_trapezoidal_rule_included} is bounded by (since $\max\{r,s\}<-t-2\alpha$)
\begin{align*}
    \tilde{\tilde{C}}\left(\|u-u_N^{(M)}\|_{H^{2\alpha}}h^{-(2\alpha+t)}+\mathcal{E}_{r,s}(\Delta_M)\left(\|u-u_N^{(M)}\|_{r+2\alpha}+\|u-u_N^{(M)}\|_{s+2\alpha}\right)\right).
\end{align*}
Now we use Corollary~\ref{cor:higher_order_sobolev_estimates} to conclude:
\begin{align*}
    \|u-u_N^{(M)}\|_{H^{t+4\alpha}}\leq \tilde{\tilde{C}}h^{l-(4\alpha+t)}\left(1+\mathcal{E}_{r,s}(\Delta_M)h^{2\alpha-t-\max\{r,s\}}\right)\|u\|_{H^{l}}
\end{align*}
for some constant $\tilde{\tilde{C}}$ independent of $N,\Delta_M,u$ and the result follows.
\end{proof}
As before we can use Lemma~\ref{lem:bound_equispaced_trapezoidal_rule_L^2_product} to make concrete predictions for equispaced collocation points $\Delta_M=\{\tilde{x}+m/M\}_{m=1}^M$ and approximation spaces $S_N$ consisting of degree $d$ splines on a quasiequispaced grid:

\begin{corollary}\label{cor:aubin-nitsche_for_splines}
If $\Delta_M$ are equispaced and $S_N$ are spline spaces of degree $d$ on a quasiequispaced mesh, and if in addition $M=M(N)\geq N^{\beta}$ for some $\beta>2+\frac{1}{2d+1-4\alpha}$, then there is a constant $C$ independent of $M,N,u$ such that 
\begin{align*}
	\|u_N^{(M)}-u\|_{H^{-d-1+4\alpha}}\leq C N^{4\alpha-2d-2}\|u\|_{H^{d+1}}.
\end{align*}
\end{corollary}
This tells us that just a bit more than quadratic oversampling suffices to achieve the fastest convergence rate in $H^{-d-1+4\alpha}$. We can infer similar results for the convergence rates in $\|\cdot\|_{H^{t+4\alpha}}$ for $-d-1<t\leq0$.

\subsection{Exact expression for the error for equispaced spline bases}\label{sec:exact_expression_error_equispaced_splines}
So far we have tried to keep the analysis fairly general to allow for suboptimal choices of collocation points. In this section we are interested in understanding the effects of oversampling when the collocation points and approximation spaces are chosen in an optimal way. We expect that oversampling can mitigate the effects of slower convergence in the standard collocation method and, with sufficient oversampling, recover the convergence rate of an associated Bubnov-Galerkin method. Specifically, in this section we choose $S_N$ to consist of splines of degree $d$ on an equispaced grid $0=x_1<x_2<\cdots<x_N<1$, with $x_n=n/N,n=1,\dots N$, and we let the corresponding collocation points be given by the following mesh refinement
\begin{align}\label{eqn:def_equispaced_matching_collocation_points}
    \Delta_M=\left\{\frac{l+\xi_j}{N}\,\Big|\, l=1,\dots N,j=1,\dots, J\right\}, \quad \xi_j=j/J,
\end{align}
such that $M=JN, J\in\mathbb{N}$, and the discrete inner product is given by
\begin{align}\label{eqn:discrete_inner_product_equispaced_points}
    \left\langle f,g\right\rangle_M=\frac{1}{N}\sum_{n=1}^N\frac{1}{J}\sum_{j=1}^J\overline{f\left(\frac{l+\xi_j}{N}\right)}g\left(\frac{l+\xi_j}{N}\right).
\end{align}
The approach we take here was first provided by Sloan \cite{sloan1988quadrature} (see also \cite{chandler1990,sloan_1992}) as a way to study generalised quadrature rules for the Galerkin inner product. In the following we adapt the results from \cite{chandler1990} to our setting, with the only minor difference between our case and theirs being that our test functions are of the form $V\chi_N$ as opposed to $\chi_N$. While Chandler \& Sloan \cite{chandler1990} were mainly focused on constructing specific quadrature rules similar to \eqref{eqn:discrete_inner_product_equispaced_points} which keep $J$ fixed as $N$ increases, the main novelty for this section is to use their arguments to understand the behaviour when $J=J(N)$ varies with $N$.

In order to facilitate this we need to make a slightly stronger assumption on $V$, namely that it has a Fourier series representation in the form
\begin{align}\label{eqn:pseudo_differential_form_V}
    Vg(x)=\hat{g}_0+\sum_{m\neq 0}[m]^{2\alpha}\hat{g}_me^{2\pi i m x},
\end{align}
where we introduced the notation 
\begin{align*}
    \left[m\right]=\begin{cases}1,&\quad\text{if\ }m=0\\
    |m|,&\quad\text{if\ }m\neq 0.
    \end{cases}
\end{align*}
This means that we assume $V$ is a pseudo-differential operator whose action maps every Fourier mode to a constant multiple of itself. An example of an operator taking this form is the single layer potential for the Laplace equation on a circular boundary.

In this case we are able to get an exact expression for the Fourier coefficients of the error $u_N^{(M)}-\tilde{u}$ and use this to derive tight estimates on the convergence rate as follows:
\begin{theorem}\label{thm:convergence_rates_equispaced_grids}
If the method satisfies the consistency condition $d>2\alpha$, then it is stable and converges satisfying the following error estimate:
\begin{align*}
    \|u_N^{(M)}-\tilde{u}\|_{H^{4\alpha-(d+1)}}\leq C\left(M^{-(d+1)+2\alpha}+N^{-2(d+1)+4\alpha}\right)\|\tilde{u}\|_{H^{d+1}}
\end{align*}
where $C$ is a constant depending on $d,\alpha$, but independent of $N,M$.
\end{theorem}
Thus we find that if $J(N)=M/N:\mathbb{N}_{>0}\rightarrow\mathbb{N}_{>0}$ the method converges like
\begin{align*}
    \tilde{u}-u_N^{(M)}=\mathcal{O}\left( M^{-({d+1})+2\alpha}+N^{-2(d+1)+4\alpha}\right)
\end{align*}
so the fastest possible rate is achieved for $M=N^2$. In particular this means that for equispaced sampling points that are a refinement of the mesh in the approximation space, linear oversampling leads to an improvement of the error by a factor of $J^{-(d+1)+2\alpha}$. For the single layer potential with linear splines this means by a factor of $J^{-3}$ which is indeed observed in practice, see Fig.~\ref{fig:linear_oversampling_graph}.
\begin{remark}\label{rmk:modified_oversampled_col_and_compact_perturbation_argument}
We mentioned in \S\ref{sec:least-squares_to_bubnov-galerkin} that one might also consider a \textit{modified} oversampled collocation method which is defined through a modified discrete orthogonality condition \eqref{eqn:modified_discrete_orthogonality_condition} of the form
\begin{align*}
     \left\langle \chi_N,Vu_N^{(M)}\right\rangle_{M}=\left\langle \chi_N,V\tilde{u}\right\rangle_{M},\quad \forall \chi_N\in S_N,
\end{align*}
and which corresponds to a discrete version of a standard Galerkin method. The analysis in this present section can be applied in an analogous way for this modified oversampled collocation method, and one can prove (under the consistency assumption $d>2\alpha$) that there is $C>0$ such that:
\begin{align}\label{eqn:optimal_convergence_modified_collocation}
    \|u_N^{(M)}-\tilde{u}\|_{H^{2\alpha-(d+1)}}\leq C\left(M^{-(d+1)+2\alpha}+N^{-2(d+1)+2\alpha}\right)\|\tilde{u}\|_{H^{d+1}}.
\end{align}
We highlight that whilst at present the results of Theorem~\ref{thm:convergence_rates_equispaced_grids} for the least-squares oversampled collocation method (as defined via \eqref{eqn:discrete_orthogonality_condition}) are restricted to operators $V$ in the pseudo-differential form \eqref{eqn:pseudo_differential_form_V}, the analysis for the modified oversampled collocation method actually extends to a compact perturbation of this form. Indeed \eqref{eqn:optimal_convergence_modified_collocation} applies equally if $V=V_0+\mathcal{K}$, where $V_0$ takes the form \eqref{eqn:pseudo_differential_form_V} and $\mathcal{K}:H^s\rightarrow H^t$ is continuous as a bounded linear map for all $s,t\in\mathbb{R}$. This follows from a standard argument that is described for instance by Arnold \& Wendland \cite[\S3]{Arnold1985} and reproduced for completeness in Appendix~\ref{app:compact_perturbation_argument}.
\end{remark}
\begin{proof}[Proof of Theorem~\ref{thm:convergence_rates_equispaced_grids}]
We follow \cite[\S2]{chandler1990} and \cite[\S7]{sloan_1992} and introduce a convenient basis for $S_N$ (where here we write $\Lambda_N=\left\{\mu\in\mathbb{Z}:-N/2<\mu\leq N/2\right\}$ and $\Lambda_N^*=\Lambda_N\setminus\{0\}$):
\begin{align*}
    \psi_\mu(x)=\begin{cases}
    1,&\mu=0\\
    \sum_{k\equiv \mu (N)}(\mu/k)^{d+1} e^{2\pi i kx},&\mu\in\Lambda_N^*.
    \end{cases}
\end{align*}
Note this is indeed a spline of the given degree $d$ since its Fourier coefficients satisfy the appropriate recurrence relation.
\begin{align*}
    k^{d+1}\hat{v}_k=\mu^{d+1}\hat{v}_\mu,\quad \text{if\ }k\equiv \mu.
\end{align*}
In many ways the basis $\{\psi_\mu\}_{\mu\in\Lambda_N}$ behaves like a Fourier basis, in particular
\begin{align*}
    \psi_{\mu}(x+n/N)=e^{2\pi i \mu n/N }\psi_{\mu}(x),
\end{align*}
which allows us to treat the leading order terms in the oversampled collocation system exactly: Let us write our oversampled collocation approximation as
\begin{align*}
    u^{(M)}_N=\sum_{\nu\in\Lambda_N}\hat{a}_\nu \psi_{\nu}
\end{align*}
and let the true solution to \eqref{eqn:original_integral_equation} be $\tilde{u}(x)=\sum_{m\in\mathbb{Z}}\hat{u}_m\exp(2\pi im x)$. The discrete orthogonality conditions \eqref{eqn:discrete_orthogonality_condition} are
\begin{align}\label{eqn:discrete_orthogonality_special_basis}
   \sum_{\nu\in\Lambda_N} \left\langle A\psi_\mu,A\psi_{\nu}\right\rangle_M\hat{a}_\nu=\left\langle A\psi_{\mu}, A\tilde{u}\right\rangle_M, \quad \mu=1,\dots N.
\end{align}
One can then show after a few steps of algebra that for $\mu\in\Lambda_N$:
\begin{align}\label{eqn:explicit_solution_linear_system}
    \hat{a}_\mu=\begin{cases}\frac{1}{J}\sum_{j=1}^J\sum_{n\equiv 0(N)}[n]^{2\alpha}\hat{u}_n\exp\left(\frac{n}{N}\xi_j\right),&\text{if\ }\mu=0\\
    D\left(\frac{\mu}{N}\right)^{-1}\frac{1}{J}\sum_{j=1}^J\sum_{n\equiv\mu (N)}\left[\frac{n}{\mu}\right]^{2\alpha}\exp\left(2\pi i \frac{n-\mu}{N}\xi_j\right) \hat{u}_n\left(1+\overline{\Omega\left(\xi_j,\frac{\mu}{N}\right)}\right),&\text{if\ }\mu\neq 0,
    \end{cases}
\end{align}
where
\begin{align*}
    D(y)&=\frac{1}{J}\sum_{j=1}^J\left|1+\Omega\left(\xi_j,y\right)\right|^2,\\
    \Omega(\xi,y)&=|y|^{{d+1}-\beta}\sum_{l\neq 0}\frac{1}{|l+y|^{{d+1}-\beta}}e^{2\pi il\xi}.
\end{align*}
In consequence it follows that
\begin{align}\label{eqn:exact_expression_error_equispaced_grids}
  \hat{a}_\mu-\hat{u}_\mu=\begin{cases}Z_N,&\text{if\ }\mu=0\\
  -\frac{E(\mu/N)}{D(\mu/N)}\hat{u}_\mu+R_N(\mu),&\text{if\ }\mu\neq 0,
  \end{cases}
\end{align}
where
\begin{align*}
    Z_N&=\sum_{\substack{n\in\mathbb{Z}\\n\neq 0}}\left[nM\right]^{2\alpha}\hat{u}_{nM},\\
        E(y)&=|y|^{d+1-2\alpha}\sum_{l\neq0}\frac{1}{|lJ+y|^{d+1-2\alpha}}+\frac{1}{J}\sum_{j=1}^J\left|\Omega(\xi_j,y)\right|^2,\\
    R_N(\mu)&=D\left(\frac{\mu}{N}\right)^{-1}\left(\sum_{k\neq 0}\left[\frac{\mu+kM}{\mu}\right]^{2\alpha}\hat{u}_{\mu+kM} \right. \\
    &\quad\quad\left. +\sum_{k\neq 0}\left[\frac{\mu+kN}{\mu}\right]^{2\alpha}\hat{u}_{\mu+kN}\left|\frac{\mu}{N}\right|^{d+1-2\alpha}\sum_{\substack{l\equiv k (J)\\l\neq 0}}\left|\frac{1}{l+\mu/N}\right|^{d+1-2\alpha}\right).
\end{align*}
The details of this derivation require only very minor modification to the discussion in \cite[Section 2]{chandler1990}, but for completeness the arguments are provided in Appendix~\ref{app:derivation_exact_error_expression}. We show in Appendix~\ref{app:derivation_exact_error_expression} in \eqref{eqn:app_proof_that_D_uniformly_bounded_away_from_0} that
\begin{align*}
    D(y)\geq 1,\quad \forall y\in[-1/2,1/2]
\end{align*}
for any choice $J=J(N):\mathbb{N}_{>0}\rightarrow\mathbb{N}_{>0}$, which means the oversampled collocation system is well-posed and stable for any integer rate of oversampling $M=J(N)N$. The expression for the error in the small frequency Fourier modes in \eqref{eqn:exact_expression_error_equispaced_grids} determines the fastest possible rate of convergence in any Sobolev norm because the following ``projection'' $P_N$ onto the low-frequencies
\begin{align*}
    P_N:f\mapsto\sum_{\mu\in\Lambda_N}\hat{f}_\mu \psi_\mu(x)
\end{align*}
satisfies
\begin{align}\label{eqn:estimate_for_small_frequency_projection}
    \|f-P_Nf\|_{H^{s}}\leq C_s N^{s-t}\|f\|_{H^{t}},\quad \forall s+1/2<t\leq d+1.
\end{align}
We can use the expressions \eqref{eqn:exact_expression_error_equispaced_grids} analogously to \cite[(7.38)-(7.39)]{sloan_1992} and show that (subject to the consistency assumption $d>2\alpha$):
\begin{align}\label{eqn:estimates_for_remainder_terms}
    |Z_N|^2&+\sum_{\mu\in\Lambda_N^*}|\mu|^{2(4\alpha-d-1)}|R_N(\mu)|^2 \\ \nonumber
    &\leq C_{\alpha,r}\left(M^{-2r+4\alpha}\|\tilde{u}\|_{H^r}^2+N^{-2r-2(d+1)+8\alpha}\|\tilde{u}\|_{H^r}^2\right),\quad \forall r>2\alpha+1/2.
    \end{align}
Thus by taking $r$ arbitrarily large we can ensure an arbitrarily fast rate of decay in the terms arising from $Z_N,R_N(\mu)$. Therefore the leading order error term is due to $E(\mu/N)$, and we can in fact estimate this one as well:
\begin{align}\label{eqn:estimates_for_leading_term}
\left|\frac{E(\mu/N)}{D(\mu/N)}\hat{u}_\mu\right|^2&\leq C_{\alpha,d}\left(\left|M\right|^{4\alpha-2(d+1)}|\mu|^{2(d+1)-4\alpha}+\left|N\right|^{8\alpha-4(d+1)}|\mu|^{4(d+1)-8\alpha}\right)\left|\hat{u}_\mu\right|,
\end{align}
if $d>2\alpha$. Thus combining \eqref{eqn:estimates_for_remainder_terms} \& \eqref{eqn:estimates_for_leading_term} we have:
\begin{align*}
\|u_N^{(M)}-P_N\tilde{u}\|_{H^{4\alpha-(d+1)}}^2&=\sum_{\mu\in\Lambda_N}\left[\mu\right]^{2(4\alpha-(d+1))}|\hat{a}_\mu-\hat{u}_\mu|^2\left(1+\sum_{l\neq0}\left[\frac{\mu}{\mu+lN}\right]^{4(d+1)-8\alpha}\right)\\
&\leq C\sum_{\mu\in\Lambda_N}\left[\mu\right]^{2(4\alpha-(d+1))}|\hat{a}_\mu-\hat{u}_\mu|^2\\
&\leq \tilde{C}\left(M^{4\alpha-2(d+1)}\|\tilde{u}\|_{H^{2\alpha}}^2+N^{8\alpha-4(d+1)}\|\tilde{u}\|_{H^{d+1}}^2\right.\\
&\quad\left.+M^{-2r+4\alpha}\|\tilde{u}\|_{H^r}^2+N^{-2r-2(d+1)+8\alpha}\|\tilde{u}\|_{H^r}^2\right),\quad \forall r>2\alpha+1/2.
\end{align*}
And combining this with the projection estimate \eqref{eqn:estimate_for_small_frequency_projection} yields the desired bound
\begin{align*}
    \|u_N^{(M)}-\tilde{u}\|_{H^{4\alpha-(d+1)}}\leq C\left(M^{2\alpha-(d+1)}+N^{-2(d+1)+4\alpha}\right)\|\tilde{u}\|_{H^{d+1}}.
\end{align*}
\end{proof}

\section{Oversampled collocation in specific settings}
\label{sec:specific_settings}

In the previous section, we focused our attention mainly on smooth boundaries with equispaced collocation points, but in this section we aim to demonstrate that a slight amount of oversampling can actually stabilise the oversampled collocation method even in the case of Lipschitz domains and the case of highly sub-optimal choices of collocation points. This is with a view to the possible advantages of oversampled collocation methods in general settings/geometries, particularly in 3D, where an optimal choice of collocation points may not be immediately obvious.
\subsection{Non-equispaced sampling points}\label{sec:non-equispaced_sampling_points}
We begin by examining non-equispaced collocation points. Let us remain in the case of a smooth boundary $\Gamma$ for this part (i.e. without loss of generality $\Gamma=[0,1)$) and consider a general sequence of collocation points
\begin{align*}
    \Delta_M=\Delta_{M(N)}=\left\{0\leq x_1<\cdots<x_M<1\right\}.
\end{align*}
We now assume nothing more than the requirement that the maximum spacing of consecutive collocation points,
\begin{align*}
    \mathrm{d}\left(\Delta_M\right)=\max_{1\leq j\leq N}|x_{j+1}-x_{j}|
\end{align*}
reduces to $0$ in a certain way as $N\rightarrow\infty$. Here it is again understood that $x_{N+1}=x_1$ and the distance is measured on the periodic domain $[0,1)$. We have the following error estimate for the discrete inner product:
\begin{lemma}\label{lem:error_estimate_trapezoidal_rule_general_sampling} Let $\Gamma, \Delta_M(N)$ be as above. Then there is a constant $C_{r,s}>0$, independent of $\Delta_M$, such that for any $f,g\in H^{\max\{r,s\}}$ with $r>5/2,s>1/2$:
    \begin{align}\label{eqn:error_estimate_trapezoidal_rule_non_equispaced}
        |\langle f,g\rangle-\langle f,g\rangle_M|\leq C_{r,s} M\mathrm{d}(\Delta_M)^3\left(\|f\|_{H^r}\|g\|_{H^s}+\|f\|_{H^s}\|g\|_{H^r}\right).
    \end{align}
\end{lemma}
\begin{proof}
    The error estimate is based on Morrey's inequality and the well-known error-expression for the trapezoidal rule for $C^2$-functions. For completeness a proof is included in Appendix~\ref{app:error_trapezoidal_rule_non-uniform_grid}.
\end{proof}
The estimate in \eqref{eqn:error_estimate_trapezoidal_rule_non_equispaced} is precisely of the form \eqref{eqn:general_error_estimate_trapezoidal_rule}, which allows us to apply Theorems \ref{thm:convergence_in_energy_space} \& \ref{thm:discrete_aubin_nitsche} for the following result:
\begin{corollary}\label{cor:convergence_estimates_for_collocation_non-equispaced_points}
If for some $\epsilon>0$ we have $5/2+\epsilon\leq m+1/2-2\alpha$, and $\Delta_M=\Delta_{M(N)}$ is chosen such that
\begin{align}\label{eqn:mesh_condition_for_energy_space_convergence}
    \lim_{N\rightarrow\infty}\mathrm{d}(\Delta_M)^3h^{-3-\epsilon}M=0,
\end{align}
then there is a constant $C>0$ independent of $M,N,u$ such that 
	\begin{align*}
		\|u_N^{(M)}-u\|_{H^{2\alpha}}\leq C h^{l-2\alpha}\|u\|_{H^{l}}.
	\end{align*}
If in addition we have for some $-l\leq t\leq 5/2+\epsilon+2\alpha$
\begin{align}\label{eqn:mesh_condition_for_Aubin_Nitsche_convergence}
    \lim_{N\rightarrow\infty}\mathrm{d}(\Delta_M)^3h^{2\alpha-t-5/2-\epsilon}=0,
\end{align}
then there is a constant $C>0$ independent of $M,N,u$ such that 
	\begin{align*}
		\|u_N^{(M)}-u\|_{H^{t+4\alpha}}\leq C h^{l-t-4\alpha}\|u\|_{H^{l}}.
	\end{align*}
\end{corollary}
These results guarantee convergence of oversampled collocation methods even when the collocation points are very badly chosen - the only condition for success is that, as $N$ increases, the collocation points are distributed sufficiently uniformly to provide a good approximation to the $L^2$-inner product in the sense of \eqref{eqn:mesh_condition_for_energy_space_convergence} \& \eqref{eqn:mesh_condition_for_Aubin_Nitsche_convergence}. Indeed one may see Corollary~\ref{cor:convergence_estimates_for_collocation_non-equispaced_points} as confirmation that, in settings when standard collocation fails, a small amount of oversampling can help resolve the convergence issues. We will see this to be the case in practice in \S\ref{sec:suboptimal_collocation_points}.
\begin{remark}
Although this is not a main focus of the current work we highlight that through a similar argument to Proposition~\ref{prop:uniform_ellipticity} it can be seen that conditioning of the weighted normal equations \eqref{eqn:general_least_squares_system} depends (if sufficiently many collocation points are taken) mainly on the properties of the basis functions $S_N$. So if the Bubnov-Galerkin method \eqref{eqn:bubnov_galerkin_method} is well-conditioned we expect that for sufficient oversampling the linear system \eqref{eqn:general_least_squares_system} is also well-conditioned.
\end{remark}

\subsection{Lipschitz domains}\label{sec:polygonal_domains}
So far we focused on smooth domains, but in this section we aim to show on a more specific case that our results in Theorem~\ref{thm:convergence_in_energy_space} and Corollary~\ref{cor:convergence_on_equispaced_grids_spline_basis} extend to Lipschitz domains. This is of interest because previous analyses of collocation methods in the literature have focused on the smooth case $\Gamma$ (e.g. \cite{arnold1983asymptotic,saranen1985asymptotic,sloan1988quadrature}). In the process of extending our results to less regular boundaries more care must be given to the continuity properties of the integral operators (similar to the analysis of any Galerkin method on similar domains) and whether the boundary Sobolev spaces are well-defined, all of which needs to be considered on a case-by-case basis. For that reason and for simplicity we shall focus on the single layer integral equation on a Lipschitz (e.g. polynomial) boundary for the Dirichlet problem of the Helmholtz equation. This means in the above notation that the parametrisation $z:[0,1)\rightarrow \Gamma$ as in \eqref{eqn:parametrisation_of_boundary} is a Lipschitz function with Lipschitz inverse (piecewise linear in the case of polygonal boundaries) and our integral equation is
\begin{align}\label{eqn:single_layer_integral_equation}
    \mathcal{S} u=f
\end{align}
where the single layer potential $S$ is as defined in Example~\ref{ex:Helmholtz_boundary_integral_formulation}. It was shown in Appendix~A.3 of \cite{chandler-wilde_graham_langdon_spence_2012} that a sensible definition of boundary Sobolev spaces for Lipschitz domains $\Gamma$ is in terms of the parametrisation $z$:
\begin{align*}
    H^{s}(\Gamma):=\left\{f\in L^2(\Gamma)\Big|f\circ z\in H^{s}([0,1))\right\},
\end{align*}
which is a consistent definition for all $0\leq s\leq 1$ (in that it has all the usual properties of Sobolev spaces, including $H^s\subset H^t$ for $s<t$, and $C^{\infty}(\mathbb{R}^2)\big|_{\Gamma}$ is dense in $H^s$). The dual spaces can be defined similarly, and we can therefore, as before, limit our attention to the spaces $H^s([0,1))$ with
\begin{align*}
    \|f\|_{H^s}&=\|f\|_{H^s([0,1))}:=\left(|\hat{f}_0|^2+\sum_{0\neq m\in\mathbb{Z}}|m|^{2s}|\hat{f}_m|^2\right)^{1/2},
\end{align*}
where
\begin{align*}
    \hat{f}_m:=\int_0^1e^{-2\pi i m t}f(t)dt,
\end{align*}
\textit{with the additional restriction} that $|s|\leq 1$. It is then shown in Theorem~2.25 of \cite{chandler-wilde_graham_langdon_spence_2012} (see also \cite[p. 1466]{torres1993}) that 
\begin{align*}
	\mathcal{S}:H^{s+1/2}(\Gamma)\rightarrow H^{s-1/2}(\Gamma)
\end{align*}
for $|s|\leq 1/2$ is bounded linear and furthermore that $\mathcal{S}$ is invertible as a mapping between the spaces indicated above for $|s| \leq 1/2$ if and only if the interior homogeneous Dirichlet problem of the associated Helmholtz equation only has the trivial solution. We note by \cite[p. 1466]{torres1993} that $\mathcal{S}:L^{2}(\Gamma)\rightarrow L^2(\Gamma)$ is compact which means that by the spectral theorem $\mathcal{S}$ must be invertible except for a countably infinite set of values for the wavenumber in the Helmholtz equation - the resonant frequencies of the interior Dirichlet problem. We now look back at the proof of Proposition~\ref{prop:uniform_ellipticity} and Theorem~\ref{thm:convergence_in_energy_space} and notice that in all estimates the only norm bounds on $V$ and its inverse appeared as 
\begin{align*}
    \|V^{-1}\|_{2\alpha\rightarrow 0},\|V\|_{r+2\alpha\rightarrow r},\|V\|_{s+2\alpha\rightarrow s}
\end{align*}
which means that as long as $0\leq r,s\leq 1$ these terms are still bounded in the Lipschitz case. Furthermore in the upper bound \eqref{eqn:upper_bound_energy_space_estimate} the any term of the form $\|u-\chi_N\|_{H^{t}}$ has $t\leq 0$ (by virtue of $r,s\leq 1=-2\alpha$). Thus, the conclusions of these two statements remain true and we have
\begin{corollary}\label{cor:convergence_estimates_Lipschitz_domains}
For $\Gamma$ Lipschitz and $V=\mathcal{S}$, suppose \eqref{eqn:general_error_estimate_trapezoidal_rule} holds for $s,r$ with $0\leq r,s<\min\{m+3/2,1\}$, and $\Delta_M=\Delta_{M(N)}$ is chosen such that
\begin{align*}
    \lim_{N\rightarrow\infty}\mathcal{E}_{r,s}(\Delta_M)h^{-(r+s)}=0,
\end{align*}
then there is a constant $C>0$ independent of $M,N,\tilde{u}$ such that for all $t\geq 0$ 
	\begin{align*}
		\|u_N^{(M)}-\tilde{u}\|_{H^{-1}}\leq C h^{\min\{l,t\}+1}\|\tilde{u}\|_{H^{t}}.
	\end{align*}
	This means the optimal convergence rate in $H^{-1}$ is achieved.
\end{corollary}
To make a more specific conclusion we can combine this with the estimate in Lemma~\ref{lem:bound_equispaced_trapezoidal_rule_L^2_product} to show that for equispaced sampling on spline spaces convergence in $H^{-1}$ is achieved if 
\begin{align*}
   M=M(N)=N^\beta,\quad \text{for\ some\ }\beta>3/2.
\end{align*}
We note moreover that it is well-known that for smooth boundary conditions on polygonal scatterers the solution $\tilde{u}$ to the Dirichlet Problem~\eqref{eqn:single_layer_integral_equation} has specific corner singularities (see Theorem~2.3 in \cite{chandler2007galerkin}), which means its regularity for arbitrary boundary conditions in $C^{\infty}_0(\mathbb{R}^2)\big|_\Gamma$ restricted to
\begin{align}\label{eqn:regularity_of_solution_polygonal_scatterers}
    \tilde{u}\in H^{s}(\Gamma)\quad \text{if\ and\ only\ if\ }s<-1/2+\pi/\max_j{\Omega_j}
\end{align}
where $\Omega_j, j=1,\dots, 2K,$ are the interior and exterior angles of the polygon.

\section{Numerical results}
\label{sec:numerical_results}

In this section we test the aforementioned theoretical results through numerical experiments. In the following examples we consider the two types of integral opertators introduced in Example~\ref{ex:Helmholtz_boundary_integral_formulation} - the single and double layer potential formulations for the Helmholtz equation as examples of integral operators with orders $2\alpha=-1$ and $2\alpha=0$ respectively.

\subsection{Smooth domains with equispaced sampling}
We begin by verifying that the results in Theorem~\ref{thm:convergence_rates_equispaced_grids} are indeed tight and accurately predict the effect of oversampling for equispaced grids and matching collocation points as defined in \eqref{eqn:def_equispaced_matching_collocation_points}. In the first example we present results for the single layer potential formulation
\begin{align*}
    \mathcal{S}u=f,
\end{align*}
where $u$ represents the normal derivative of the field $\partial_n\phi$ on the boundary $\Gamma$. We choose $\Gamma=\{|x|=1\}$, the unit circle, and wavenumber $k=4.2$. For this domain an exact reference solution $\tilde{u}$ in terms of Bessel functions is available (cf. pp. 501--504 in \cite{morse1953methods}, see also \cite{weisstein2021}). Therefore we can evaluate the error directly in the Sobolev norms $\|\cdot\|_{H^{\Gamma}}$ which we compute using the expression in terms of Fourier modes as in \eqref{eqn:definition_sobolev_norms_in_terms_of_fourier_modes}. In Fig.~\ref{fig:sobolev_errors_for_smooth_case} we display the Sobolev error $\|\tilde{u}-u_N\|_{H^s}$ and the error in a field point $|\mathcal{S}(\tilde{u}-u_N^{(M)})(x)|$ (the latter of which converges at the fastest rate of any Sobolev error as per \eqref{eqn:estimate_field_point_error_in_terms_of_sobolev_norms}) for the following quantities:
\begin{itemize}
    \item $H^s$-projection: The orthogonal projection of $\tilde{u}$ onto $S_N$ with respect to the inner product $\langle\cdot,\cdot \rangle_s$.
    \item Galerkin method: The solution of the continuous Galerkin equations
    \begin{align*}
        \langle \chi_N,Vu_n\rangle=\langle\chi_N,f\rangle,\quad \forall\chi_N\in S_N.
    \end{align*}
    \item Collocation - $M=N$: The standard collocation method at equispaced points, $x_m=m/M$.
    \item Collocation - $M=N\lfloor N^{1/2}\rfloor$, $M=N^2$: The oversampled collocation method at equispaced points with the appropriate rates of oversampling.
\end{itemize}

\begin{figure}[t!]
	\centering
	\begin{subfigure}{0.49\textwidth}
				\centering
		\includegraphics[width=0.99\linewidth]{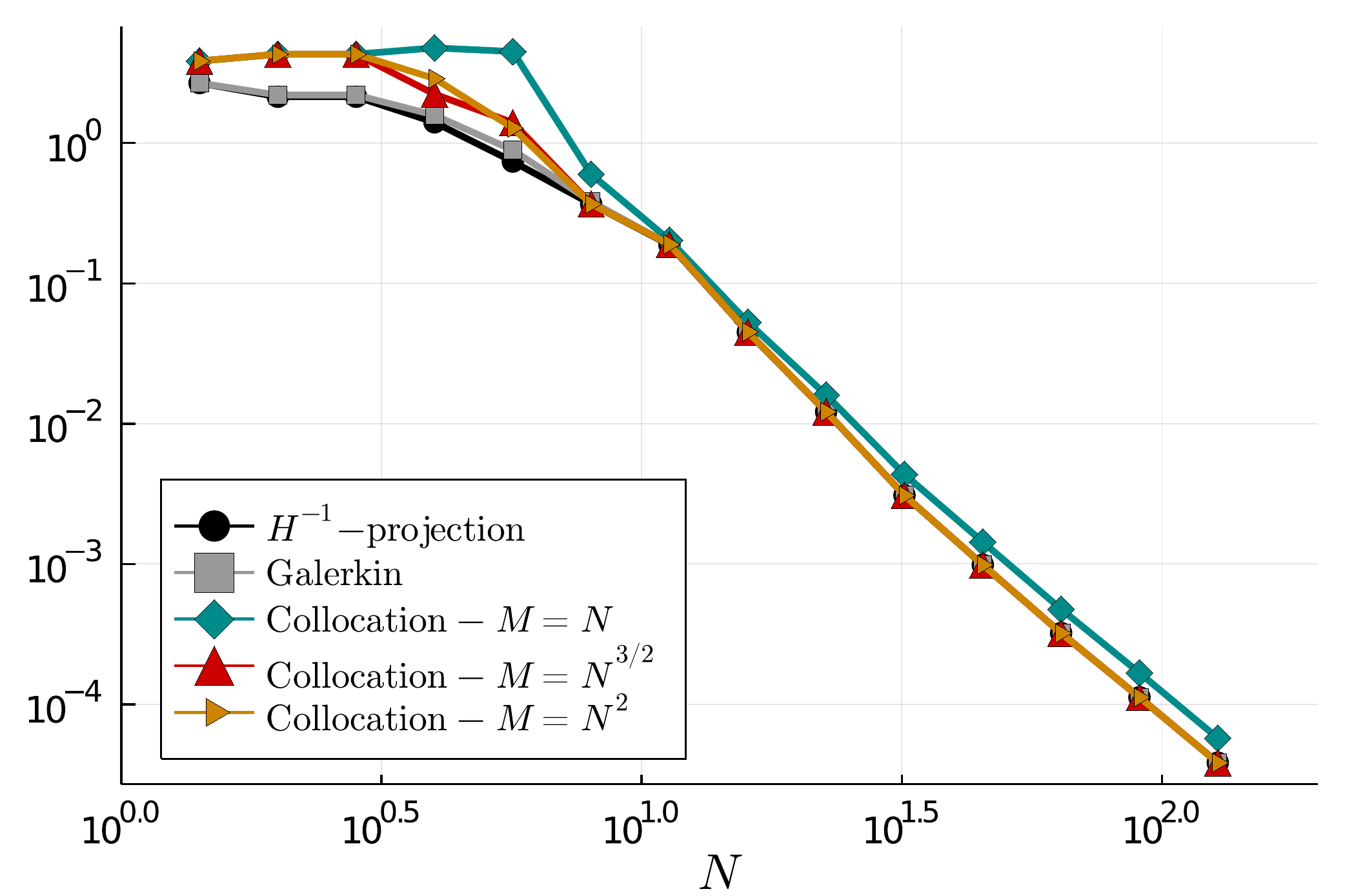}
		\caption{$\|u-u_N^{(M)}\|_{H^{-1}}$.}
		\label{fig:sobolev_errors_for_smooth_case_s-1}
	\end{subfigure}%
	\begin{subfigure}{0.49\textwidth}
		\centering
		\includegraphics[width=0.99\linewidth]{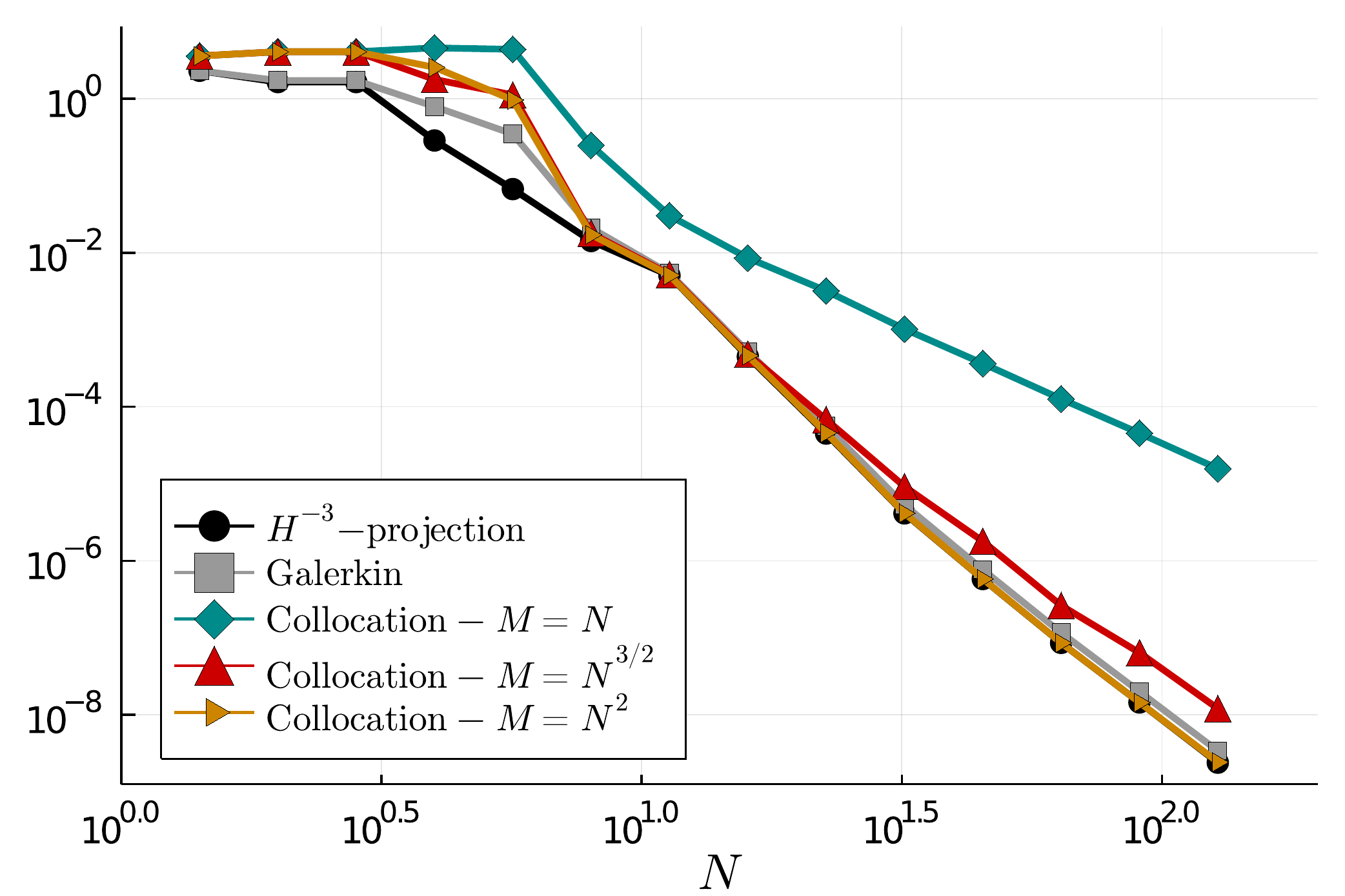}
		\caption{$\|u-u_N^{(M)}\|_{H^{-3}}$.}
		\label{fig:sobolev_errors_for_smooth_case_s-3}
	\end{subfigure}\\
	\begin{subfigure}{0.49\textwidth}
		\centering
\includegraphics[width=0.99\linewidth]{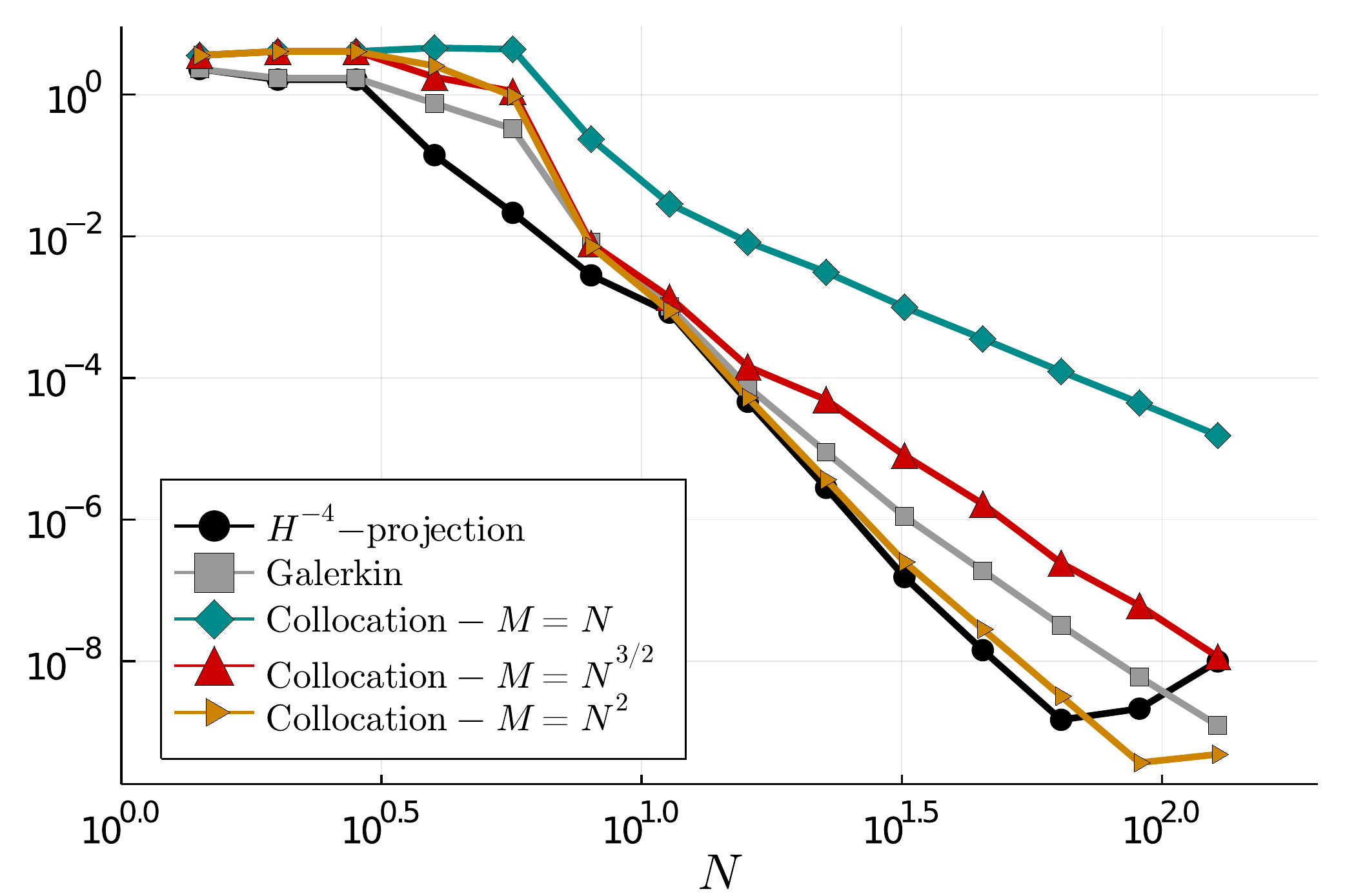}
\caption{$\|u-u_N^{(M)}\|_{H^{-4}}$.}
\label{fig:sobolev_errors_for_smooth_case_s-4}
	\end{subfigure}%
		\begin{subfigure}{0.49\textwidth}
		\centering
\includegraphics[width=0.99\linewidth]{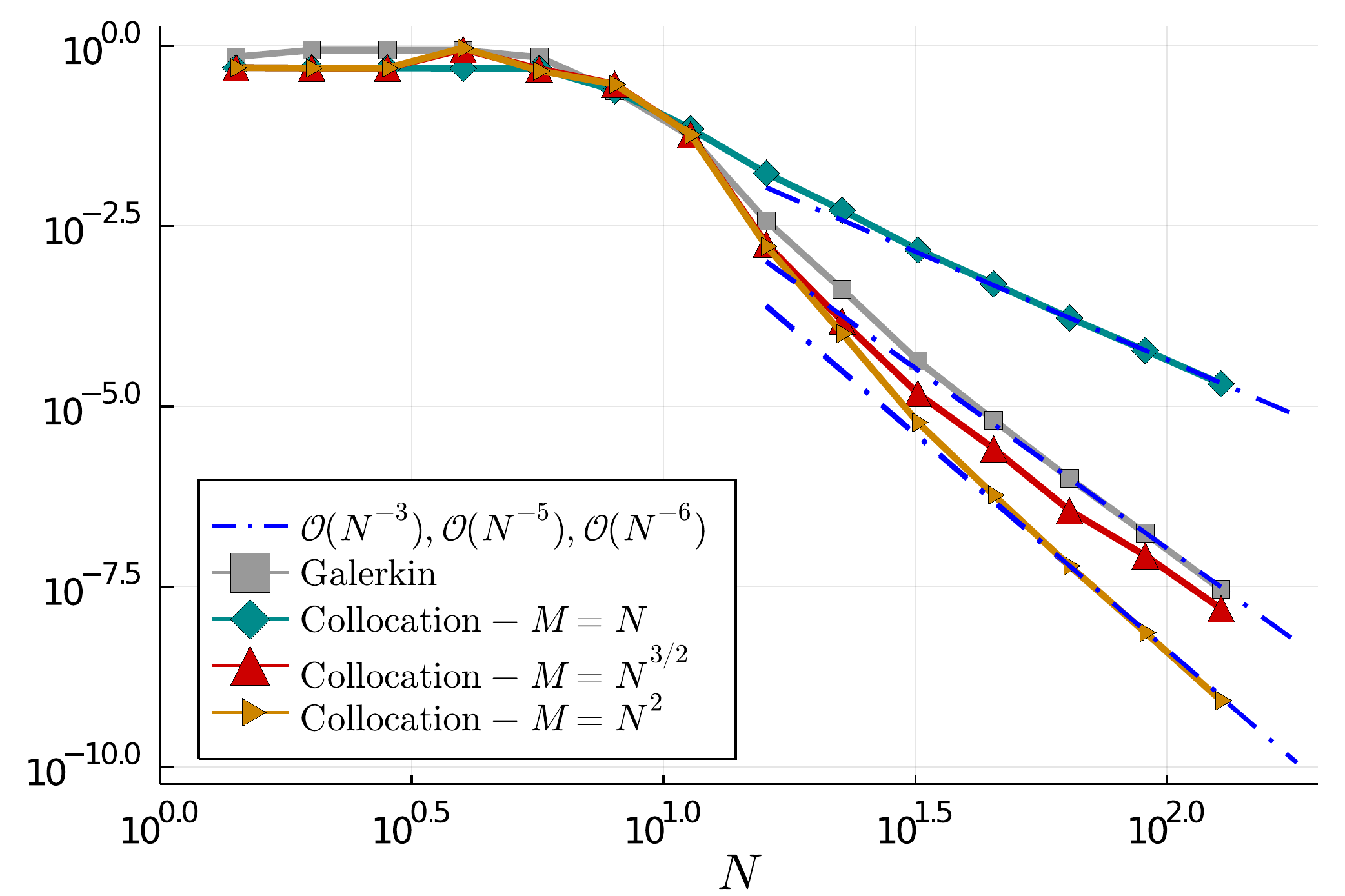}
\caption{Error in a field point $|\mathcal{S}(\tilde{u}-u_N^{(M)})(x)|$.}
\label{fig:error_field_point_scattering_on_circle}
	\end{subfigure}\\
	\caption{Error in the numerical method $\|u-u_N^{(M)}\|_{H^{s}}$ for a smooth circular scatterer, using the single layer potential (order $2\alpha = -1$) and linear splines ($d=1$). Convergence is seen in Sobolev spaces of low order, with the fastest convergence occurring at $s=4\alpha -(d+1)=-4$. This rate is achieved with oversampled collocation and quadratic oversampling ($M=N^2$), but not with Galerkin. In (\subref{fig:error_field_point_scattering_on_circle}) the error in a field point shows the fastest convergence rate in any Sobolev norm, the blue dash-dotted lines indicate, from top to bottom, $\mathcal{O}(N^{-3})$, $\mathcal{O}(N^{-5})$  and $\mathcal{O}(N^{-6})$ respectively.\vspace{-0.1cm} }\label{fig:sobolev_errors_for_smooth_case}
\end{figure}

We use splines of degree $1$, which means in terms of earlier notation $d=1$, ${2\alpha=-1}$. We now refer to the results in \S\ref{sec:exact_expression_error_equispaced_splines}, where we showed that for operators of the pseudo-differential form \eqref{eqn:pseudo_differential_form_V} the convergence rates of the oversampled collocation method are
\begin{align*}
    \mathcal{O}\left(M^{-(d+1)+2\alpha}+N^{-2(d+1)+4\alpha}\right)=\mathcal{O}\left(M^{-3}+N^{-6}\right),
\end{align*}
and that the fastest possible rate is attained in $H^{4\alpha-(d+1)}$. Indeed, the same reasoning as in the proof of Theorem~\ref{thm:convergence_rates_equispaced_grids} actually shows the finer result that the convergence order for the oversampled collocation method in $H^{t}, 4\alpha-(d+1)\leq t\leq 2\alpha$ is
\begin{align*}
    \mathcal{O}\left(M^{-(d+1)+2\alpha}+N^{\min\{t-(d+1),-2(d+1)+4\alpha\}}\right)=\mathcal{O}\left(M^{-3}+N^{\min\{t-2,-6\alpha\}}\right).
\end{align*}
In our current numerical results the single layer potential $\mathcal{S}$ is of the form $\mathcal{S}=\mathcal{S}_0+\mathcal{K}$ where $\mathcal{S}_0$ has the form \eqref{eqn:pseudo_differential_form_V} and $\mathcal{K}$ is an integral operator with smooth kernel, meaning $\mathcal{K}:H^{s}\rightarrow H^t$ for any $s,t\in\mathbb{R}$. \newpage In Fig.~\ref{fig:sobolev_errors_for_smooth_case} we observe that the predictions of Theorem~\ref{thm:convergence_rates_equispaced_grids} seem to also apply in the present case (where the optimal convergence rates, i.e. the rates of convergence of the $H^s$-projections, are included in the thin blue dashdotted lines -- these correspond to $\mathcal{O}(N^{-3}),\mathcal{O}(N^{-5}),\mathcal{O}(N^{-6})$ respectively):
\begin{itemize}
    \item In the energy space, $H^{-1}$, the optimal convergence rates are achieved for any choice of $M= JN$, $J=J(N)\geq 1$, in particular for the standard collocation method ($J=1$), which reflects the results from \cite{arnold1983asymptotic}.
    \item In the space $H^{-3}$ the Galerkin method achieves its fastest rate of convergence $\mathcal{O}(N^{-5})$, as does the oversampled collocation method with $M=N^{1.5}$. Indeed, we also see that the standard collocation method converges at a slower rate in this norm, as expected.
    \item Finally, in the space $H^{-4}$ the oversampled collocation method with $M=N^2$ converges at the optimal rate $\mathcal{O}(N^{-6})$ as predicted by the results in \S\ref{sec:exact_expression_error_equispaced_splines}, whereas all of the other methods including Galerkin converge at slower rates.
\end{itemize}
Although we highlight that the conditions for Theorem~\ref{thm:convergence_rates_equispaced_grids} are not satisfied in the current case, our more general results in Corollaries~\ref{cor:convergence_on_equispaced_grids_spline_basis}-\ref{cor:aubin-nitsche_for_splines} do apply and would guarantee that the observed convergence rates are indeed achieved when $M\geq N^\beta$ with slightly more oversampling than is used in the example:
\begin{itemize}
    \item $\beta>1+\frac{1}{2d+1-4\alpha}=1+1/5$ for optimal convergence in $H^{-1}$,
    \item $\beta>2+\frac{1}{2d+1-4\alpha}=2+1/5$ for optimal convergence in $H^{-4}$.
\end{itemize}
\begin{remark}
We note in Fig.~\ref{fig:sobolev_errors_for_smooth_case_s-4} that the convergence rates of our methods appear to level off around $N\approx 10^2$. This is a result of numerical errors due to the ill-conditioning in the $H^{-4}$-projection matrix and the need to compute a large number of Fourier coefficients of the error function to very high accuracy to accurately compute the Sobolev norm.
\end{remark}

\begin{figure}[h!]
		\centering
\includegraphics[width=0.6\linewidth]{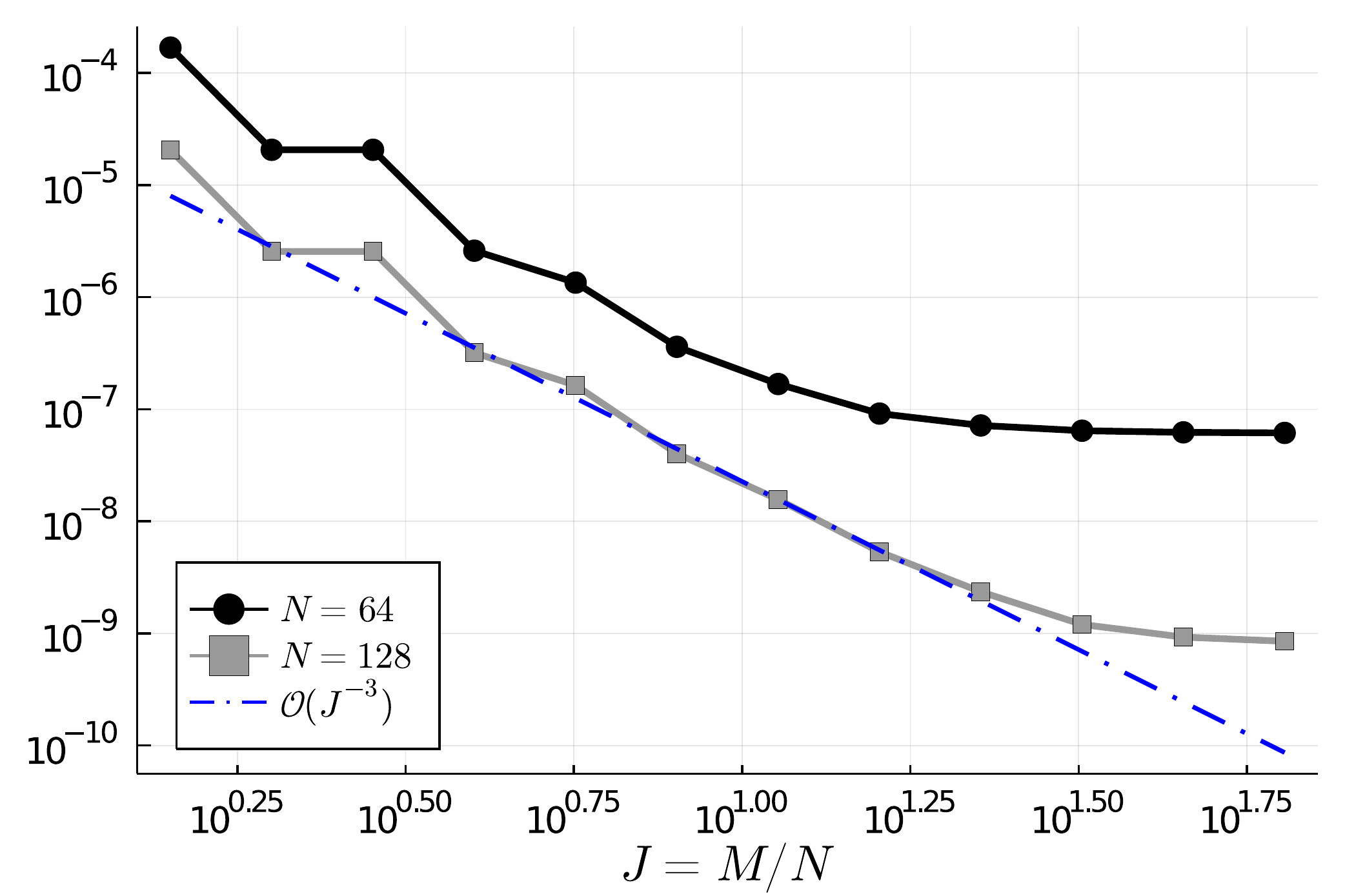}
\caption{The effect of linear oversampling for the same experiment as in Fig.~\ref{fig:sobolev_errors_for_smooth_case}. Here, $N=64,128$ and $M=JN$. For linear splines ($d=1$), \eqref{eqn:estimates_for_leading_term} predicts $O(J^{-(d+1)+2\alpha})=O(J^{-3})$ behaviour, for $1\leq J\lesssim N$.}
\label{fig:linear_oversampling_graph}
\end{figure}

Of course, faster than linear oversampling is unlikely to yield methods which are overall computationally competitive, though it cannot be ruled out a-priori that structured low-rank approximations to the system matrix in \eqref{eqn:discrete_orthogonality_condition} may (partially) offset the increased dimension of the matrix. Still, we are interested in the quantification of the benefit of linear oversampling. We highlight that, for $N$ sufficiently large, Theorem~\ref{thm:convergence_rates_equispaced_grids} actually has a consequence for linear oversampling as well. Indeed, if we fix $N$ sufficiently large and we choose $M=JN$ for a constant $J\in\mathbb{N}$, then the theorem predicts that in an initial range $1\leq J\lesssim N$ increasing $J$ results in a decay of the error of order $\mathcal{O}(J^{-(d+1)+2\alpha})=\mathcal{O}(J^{-3})$. This means, whilst linear oversampling improves the overall error just by a constant, the improvement is cubic in $J$ and so still worthwhile. The result is observed in practice as shown in Fig.~\ref{fig:linear_oversampling_graph}.

\subsection{Suboptimal collocation points}\label{sec:suboptimal_collocation_points}
Having understood the effect of oversampling in improving the convergence rate of the collocation method for an optimal choice of collocation points, we turn our attention to the case of suboptimal choices of these points. For the present examples we consider the double layer formulation of the Helmholtz equation on a smooth domain, i.e.
\begin{align*}
    V=\frac{1}{2}\mathcal{I}+\mathcal{D}
\end{align*}
with order $2\alpha=0$. We now consider two examples for the interior Dirichtlet problem on the kite shape shown in Fig.~\ref{fig:intro_example_singlelayeroversamplingfieldpointkite}, which is parametrised by
\begin{align*}
    z:t\mapsto (-\sin(2\pi t)-\cos(4\pi t),\cos(2\pi t)).
\end{align*}
In both cases we plot the error in a field point for the interior field, which in this case is given by
\begin{align*}
    |\tilde{\phi}(x)-\phi_N^{(M)}(x)|=\left|\mathcal{D}\left(\tilde{u}(y)-u_N^{(M)}(y)\right)(x)\right|=\left|\int_\Gamma\frac{\partial G}{\partial n_y}(x,y)\left(\tilde{u}(y)-u_N^{(M)}(y)\right)ds_y\right|
\end{align*}
and which according to \eqref{eqn:estimate_field_point_error_in_terms_of_sobolev_norms} captures the optimal convergence properties of the method in any Sobolev norm. We solve the interior Dirichlet problem with the field point $x=(0.1,0.2)$, wavenumber $k=5$ and plane wave boundary conditions
\begin{align*}
    \phi\big|_\Gamma(x_1,x_2)=e^{i\cos\theta x_1+i\sin\theta x_2},
\end{align*}
with $\theta=0$, for which an exact solution of the interior problem is given precisely in terms of a plane wave $\phi(x)=\phi(x_1,x_2)=\exp\left(i\cos\theta x_1+i\sin\theta x_2\right)$.

\textbf{The first example} concerns linear splines ($d=1$) on an equispaced mesh, but we take collocation points that are slightly offset. In particular we take
\begin{align*}
    \Delta_M=\left\{0.5/N+m/M\big|m=1,\dots,M\right\},
\end{align*}
i.e. for $M=N$ the collocation points are the midpoints of the spline mesh and for higher rates of oversampling the collocation points are shifted by $0.5/N$. The results are shown in Fig.~\ref{fig:equispaced_shifted_grid_kite}.

\begin{figure}[h!]
	\centering
	\includegraphics[width=0.6\linewidth]{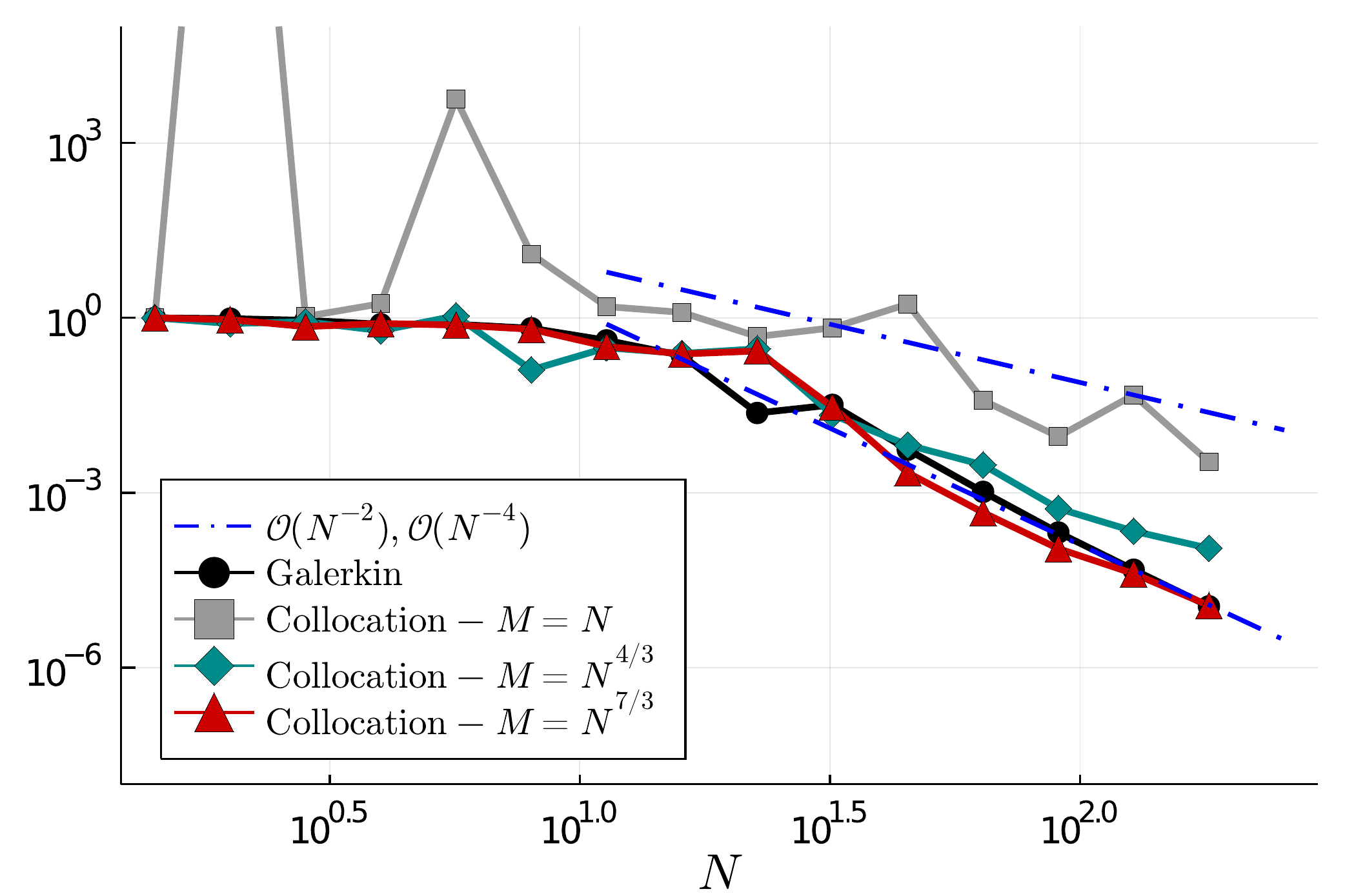}
	\caption{Error in an interior field point $\left|\mathcal{D}\left(\tilde{u}(y)-u_N^{(M)}(y)\right)(x)\right|$ for wave scattering on a smooth domain, with equispaced points that are offset from the equispaced spline mesh.}
	\label{fig:equispaced_shifted_grid_kite}
\end{figure}

We find that, as guaranteed by Corollary~\ref{cor:convergence_on_equispaced_grids_spline_basis}, there is convergence at rate $\mathcal{O}(N^{-2})$ for slightly more than linear oversampling $M=N^{\beta},$ with $\beta=1+\frac{1}{2d+1-4\alpha}=4/3$ and convergence at rate $\mathcal{O}(N^{-4})$ for slightly more than quadratic oversampling $M=N^{\beta}$, with $\beta=2+\frac{1}{2d+1-4\alpha}=7/3$. These two rates are indicated with the blue dashdotted lines in Fig.~\ref{fig:equispaced_shifted_grid_kite}.

We notice that the standard collocation method appears to converge too, albeit at a much more unsteady rate with significant problems for small $N$. This is not currently captured by our estimates and might mean that smaller amounts of oversampling can achieve the desired results, however we highlight that overall one may infer that a slight amount of oversampling can help enhance the robustness of the method to the choice of sampling points.

The positive effect of oversampling is even more noticable in our \textbf{second example} where we choose the collocation points in a highly suboptimal way -- we draw the points independently from a uniform random distribution:
\begin{align}\label{eqn:random_collocation_points}
    y_m\sim U[0,1),\quad m=1,\dots, M,
\end{align}
and $\{x_m\}_{m=1}^M=\{y_m\}_{m=1}^M$ with $0\leq x_1<x_2<\cdots<x_M<1$. In this example we take splines of degree $d=2$ (again on an equispaced mesh) to satisfy the assumptions of Corollary~\ref{cor:convergence_estimates_for_collocation_non-equispaced_points}, but other than that consider the same integral equation as in the previous example of Fig.~\ref{fig:equispaced_shifted_grid_kite}.
\begin{figure}[h!]
	\begin{subfigure}{0.49\textwidth}
		\centering
		\includegraphics[width=0.97\linewidth]{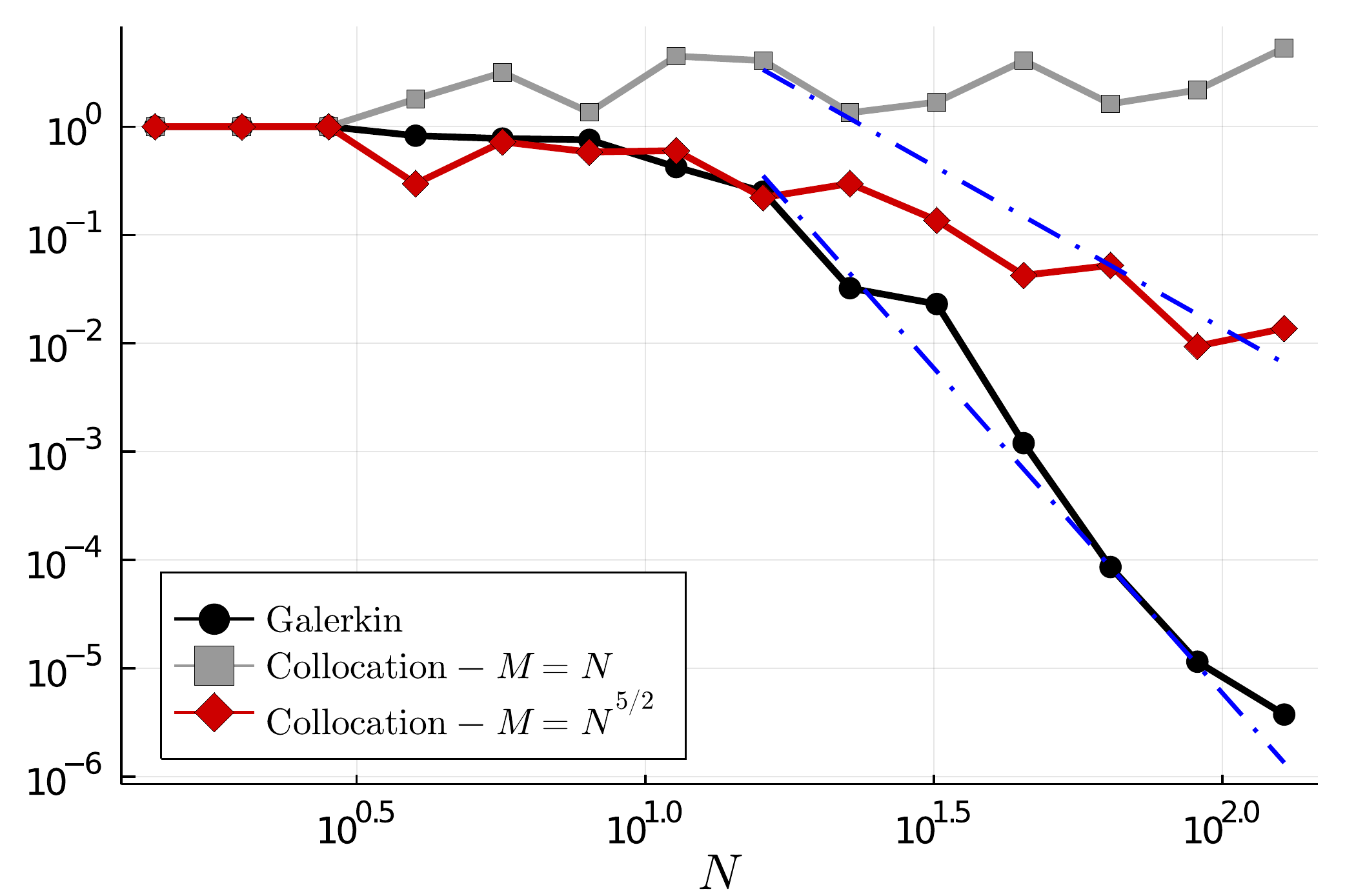}
		\caption{$\left|\mathcal{D}\left(\tilde{u}(y)-u_N^{(M)}(y)\right)(x)\right|$}
		\label{fig:random_sampling_kite-1}
	\end{subfigure}%
	\begin{subfigure}{0.49\textwidth}
		\centering
		\includegraphics[width=0.97\linewidth]{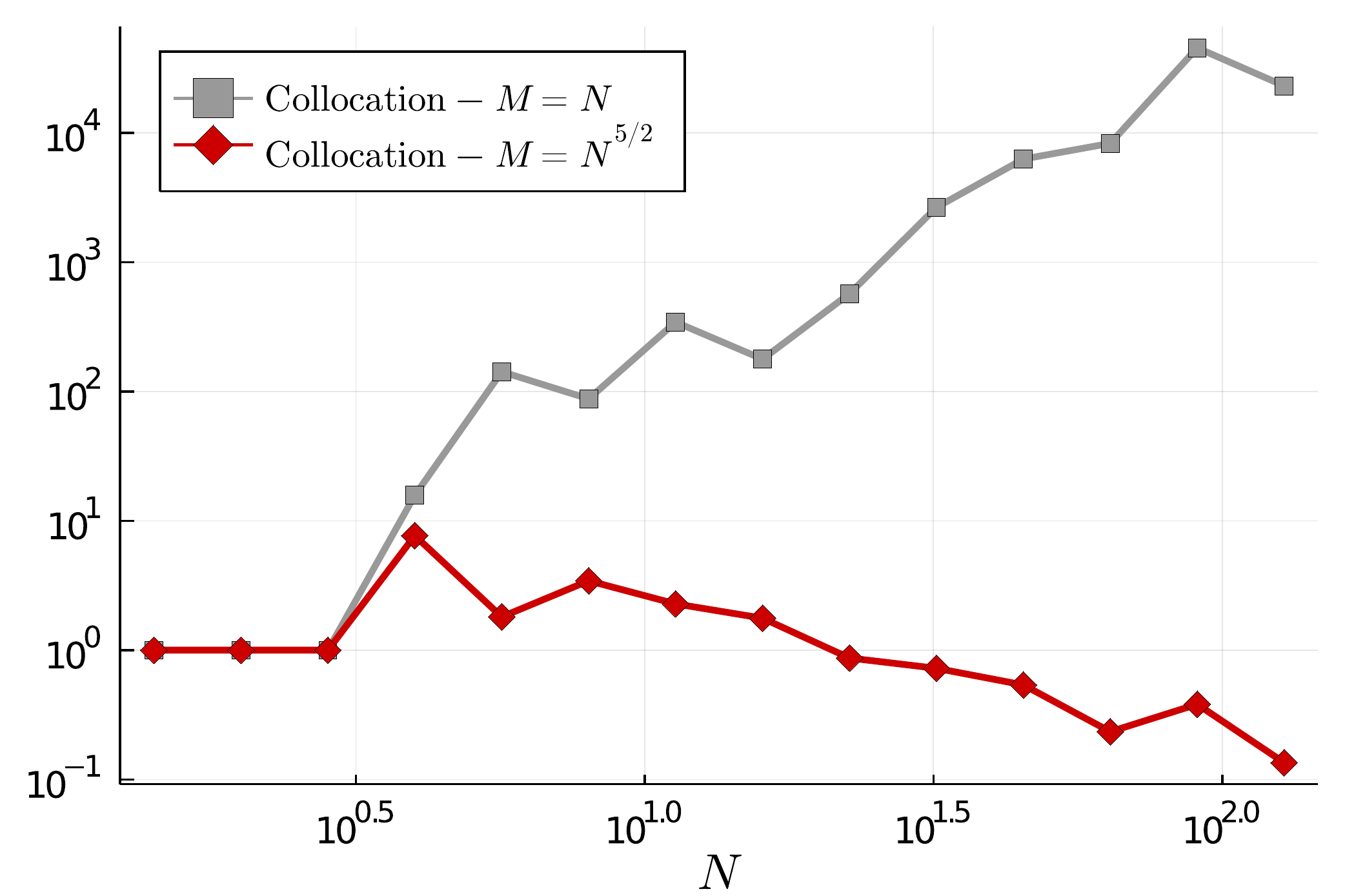}
		\caption{$\mathrm{d}(\Delta_M)^3MN^3$}
		\label{fig:random_sampling_kite-2}
	\end{subfigure}
	\caption{Double layer formulation of the interior Dirichlet problem for the Helmholtz equation. The sampling points are drawn uniformly at random $x_m\sim U[0,1)$.}
	\label{fig:random_sampling_kite}
\end{figure}

 As before in \S\ref{sec:non-equispaced_sampling_points} we let $\mathrm{d}(\Delta_M)=\max_{1\leq  m\leq M}|x_{m+1}-x_m|$. One can then show that (see e.g. Theorem~2.2 in \cite{holst1980}) the expected maximum distance in the collocation points is
\begin{align}\label{eqn:expected_value_max_distance}
    \mathbb{E}\left[\mathrm{d}(\Delta_M)\right]=\frac{1}{M-1}\sum_{m=1}^{M-1}\frac{1}{m}\leq \frac{\log (M-1)}{M-1}.
\end{align}
We recall from Corollary~\ref{cor:convergence_estimates_for_collocation_non-equispaced_points} that convergence in the energy space $H^{2\alpha}=H^{0}$ (since for the double layer potential $2\alpha=0$) is guaranteed if $\Delta_M=\Delta_{M(N)}$ is chosen such that for some $\epsilon>0$
\begin{align}\label{eqn:convergence_condition}
    \lim_{N\rightarrow 0}\mathrm{d}(\Delta_M)^3 MN^{3+\epsilon}=0.
\end{align}
Based on \eqref{eqn:expected_value_max_distance} we expect \eqref{eqn:convergence_condition} to be the case whenever $M\geq N^{\beta},\beta>4/3$. However due to the $\log$-term in \eqref{eqn:expected_value_max_distance} this convergence may only occur for rather large values of $N$. As such in the present example we increase the amount of oversampling (still in the form \eqref{eqn:random_collocation_points}) for our experiment to $M=N\lceil N^{1.5}\rceil\geq N^{2.5}$. The results are shown in Fig~\ref{fig:random_sampling_kite}. In Fig.~\ref{fig:random_sampling_kite-2} one can see the convergence quantity $\mathrm{d}(\Delta_M)^3MN^3$ which is seen to tend to zero for the oversampled collocation method, but which diverges for the case $M=N$
i.e. for the standard collocation method. In Fig.~\ref{fig:random_sampling_kite-1} we see the error in the field point $x=(0.1,0.2)$ for the Galerkin method, the collocation method with $M=N$ and for the oversampled collocation method with $M=N^{5/2}$. We notice that the standard collocation method does not converge. In contrast the oversampled collocation method exhibits convergence and the predicted asymptotic rate $\mathcal{O}(N^{2\alpha-(d+1)})=\mathcal{O}(N^{-3})$ starts to emerge for large $N$, although due to the random nature of the samples the convergence behaviour is slightly more inconsistent than in previous examples. In the plot this asymptotic rate is highlighted by the upper blue dashdotted line. Of course, as expected the Galerkin method converges at the rate $\mathcal{O}(N^{2\alpha-2(d+1)})=\mathcal{O}(N^{-6})$. This faster rate could be achieved with higher rates of oversampling in the collocation method, but this would become impractical very quickly as we increase $N$.

\subsection{Polygonal scatterers}
Our final numerical example concerns a polygonal domain. Specifically we consider the single layer formulation for the exterior scattering problem on a pentagonal domain as shown in figure \ref{fig:pentagon_geometry-field-point}.
We solve the integral equation using a spline basis of degree 1 based on the piecewise parametrisation $z:[0,1)\rightarrow\Gamma$ of the pentagon. We showed in Corollary~\ref{cor:convergence_estimates_Lipschitz_domains} that taking
\begin{align*}
    M=N^\beta,\quad \beta>3/2
\end{align*}
is sufficient to guarantee convergence at optimal rate in the energy space $H^{-1}$. For the present example we offset the collocation points by a quarter of the spline basis mesh, i.e. 
\begin{align*}
    \Delta_M=\left\{0.25/N+m/M\big|m=1,\dots,M\right\},
\end{align*}
to emphasize that the results are not dependent on an optimal choice of collocation points even for polygonal domains. The geometry is a regular polygon of sidelength $2\sin\frac{2\pi}{5}$ and we take $k=10$. In Fig.~\ref{fig:pentagon_shifted_grid_sobolev_error} the error for standard collocation and oversampled collocation methods are compared to the Galerkin method and the optimal rate of convergence provided by the $H^{-1}$-projection. We observe that all three methods follow the optimal rate very closely, albeit the oversampled collocation method does so with a smaller constant than the standard collocation. Of course, it seems even the standard collocation method converges which suggests that already smaller amounts of oversampling may be beneficial and guarantee convergence.

\begin{figure}[h!]
		\centering
\includegraphics[width=0.6\linewidth]{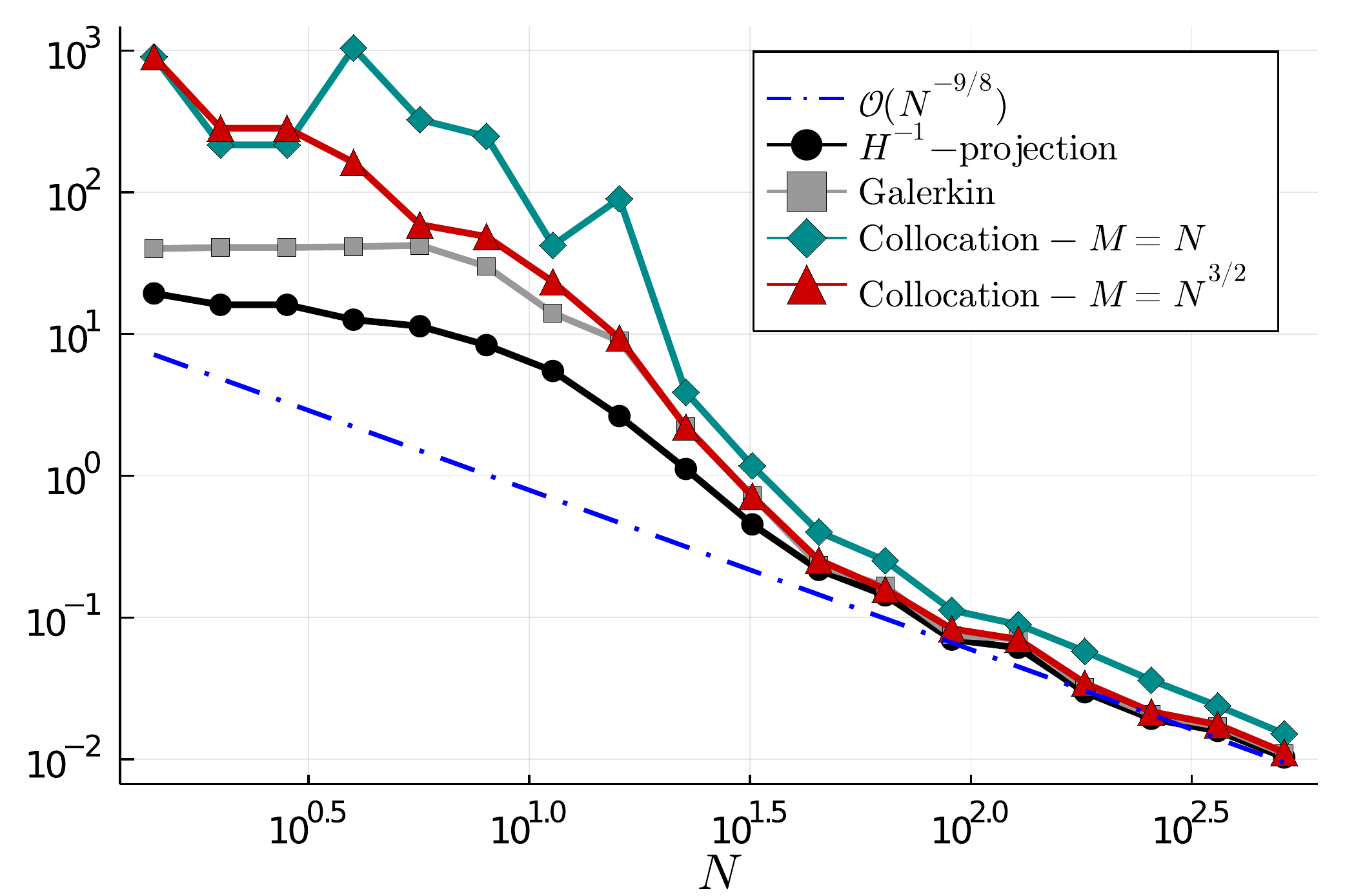}
\caption{Error in the numerical method $\|u-u_N^{(M)}\|_{H^{-1}}$ for a regular pentagonal scatterer, using the single layer potential (order $2\alpha = -1$) and linear splines ($d=1$), with offset equispaced collocation points.}
\label{fig:pentagon_shifted_grid_sobolev_error}
\end{figure}
Although there is no guarantee in the polygonal case as to whether the energy space provides the lowest Sobolev order for which optimal convergence holds we highlight in Fig.~\ref{fig:pentagon_shifted_grid_field-point_error} the convergence rates of the method in a field point, i.e. we plot
\begin{align*}
    |\mathcal{S}(\tilde{u}-u_N^{(M)})(x)|.
\end{align*}
The blue dashdotted line indicates the optimal convergence rate in $H^{-1}$, which is
\begin{align*}
    \mathcal{O}(N^{-1/2-\pi/\max_j\Omega_j})=\mathcal{O}(N^{-9/8})
\end{align*}
according to \eqref{eqn:regularity_of_solution_polygonal_scatterers}. It appears that the convergence rate in the field point is very close to this rate, and we also observe again that, even though standard and oversampled collocation methods appear to converge at similar rates, the latter does so with a smaller constant and improved stability for small values of $N$.

\begin{figure}[h!]
\begin{subfigure}{0.65\textwidth}
	\centering
	\includegraphics[width=0.99\linewidth]{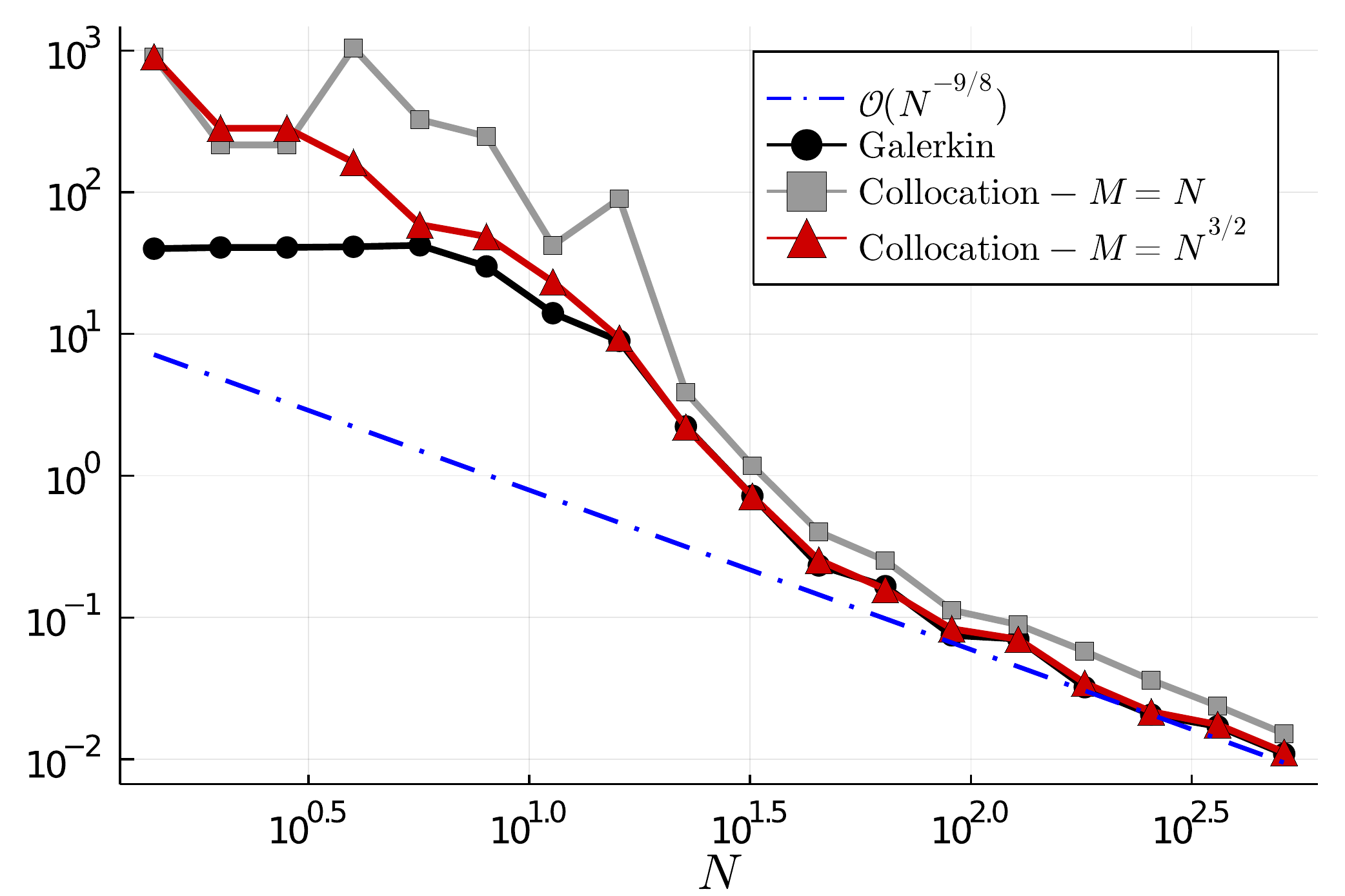}
	\caption{Convergence rates in the field point $|\mathcal{S}(\tilde{u}-u_N^{(M)})(x)|$.}
\label{fig:pentagon_shifted_grid_field-point_error}
	\end{subfigure}%
	\begin{subfigure}{0.35\textwidth}
	\centering
	\vspace{0.5cm}
	\includegraphics[width=0.99\linewidth]{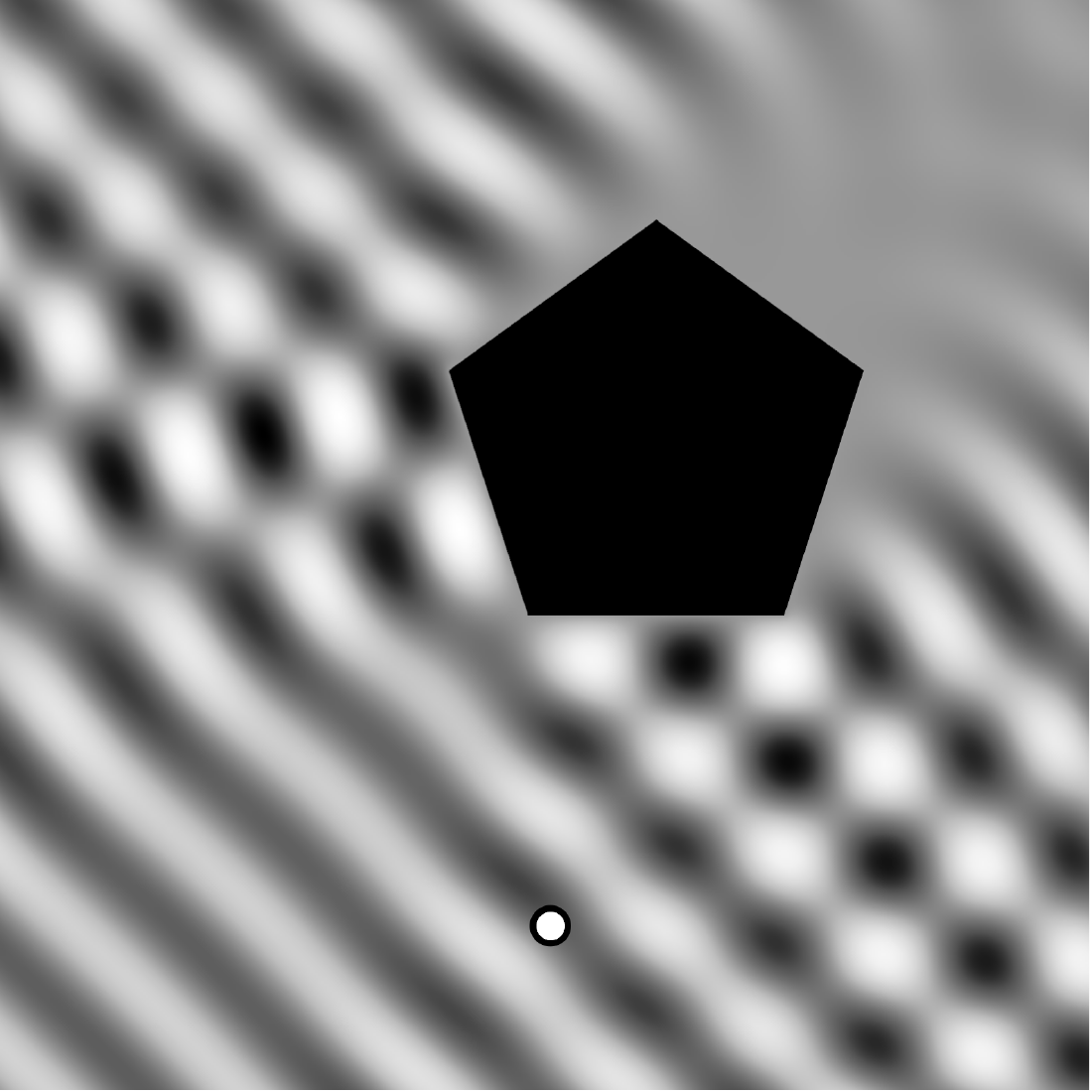}
	\vspace{0.225cm}
	\caption{The geometry and field point.}
	\label{fig:pentagon_geometry-field-point}
	\end{subfigure}
	\caption{Single layer formulation of the exterior Dirichlet problem for the Helmholtz equation on a pentagonal scatterer.}
\label{fig:pentagon_shifted_grid}
\end{figure}

\section{Conclusions and future research}\label{sec:conclusions}
In the present work, we considered \textit{oversampled collocation methods} for Fredholm integral equations, i.e. collocation methods for which the number of collocation points is greater than the dimension of the trial space. Specifically, in the method of consideration, the approximation is given as the least-squares solution to the overdetermined linear system arising from this oversampling process. Our goal was to understand whether this process can be used to enhance the approximation in the collocation setting, with a particular emphasis on avoiding having to choose the collocation points in a very specific, optimal way.

We sought to understand the properties of the oversampled collocation method both through rigorous analysis and numerical examples and our conclusions are twofold. Firstly, we showed that for sufficient amounts of (superlinear) oversampling the convergence rate of the collocation method can be improved using oversampling. Indeed in the limited number of cases where rigorous results for collocation methods do exist the convergence rate is found to be $\mathcal{O}(N^{-(d+1)+2\alpha})$ where $N$ and $d$ are the dimension and degree of the spline approximation space respectively and $2\alpha$ is the order of the integral operator. We show that due to the relationship of the oversampled collocation method with a Bubnov-Galerkin method, superlinear sampling can double the convergence rate to $\mathcal{O}(N^{-2(d+1)+4\alpha})$. Our main results in this direction are Theorems~\ref{thm:discrete_aubin_nitsche}\&\ref{thm:convergence_rates_equispaced_grids}. The former is formulated in a very general framework, based on a general selection of collocation points and regular boundary element spaces in the sense of Babu\v{s}ka \& Aziz \cite{babuska1992}, and provides a sufficient rate of oversampling to guarantee these improved convergence rates. Theorem~\ref{thm:convergence_rates_equispaced_grids} is specific to settings where the integral operator takes a certain pseudo-differential form which allows us to give an exact expression for the error and to show that in the corresponding setting the fastest convergence rate is achieved in the oversampled collocation method when $M\propto N^2$. Of course, oversampling at a quadratic rate may not be favourable in practice but the same results show that in some cases even linear oversampling at rate $M=JN$ can decrease the error of the method by a factor of $J^{-3}$, which seems certainly worthwhile.

Our second conclusion is that oversampling can help enhance the robustness of the method to a suboptimal choice of collocation points. The main result in this direction is Theorem~\ref{thm:convergence_in_energy_space} which provides a convergence guarantee for the oversampled collocation method for a very general choice of collocation points and trial spaces. Indeed this result quantifies a sufficient amount of oversampling that guarantees convergence of the method even for highly suboptimal choices of collocation points. This result is strongly visible in Fig.~\ref{fig:random_sampling_kite} where a suboptimal choice of collocation points leads the standard collocation method to diverge, whilst the oversampled version converges. Additionally, in Fig.~\ref{fig:equispaced_shifted_grid_kite} and Fig.~\ref{fig:pentagon_shifted_grid} we provided numerical evidence that especially for small values of $N$ (i.e. the initial range in the convergence plots) oversampling has a significant stabilising effect on the error of the collocation method.

Certainly, in the settings that were considered in this paper an optimal choice of collocation points is readily available, but we regard the present results as providing an initial analysis of the underlying mechanism with a view to more general settings such as integral equations on surfaces in 3D where an optimal choice may no longer be obvious or known. Indeed we highlight again that our results in Theorems~\ref{thm:convergence_in_energy_space}\&\ref{thm:discrete_aubin_nitsche} are formulated in a general way that allows for simple extension to higher dimensional domains. Future research on this topic will focus on exploring the properties of oversampling for collocation methods in more general settings. This includes in particular the extension of the present results to 3D boundary integral equations and also to understanding the effects of oversampling when there is redundancy in the trial spaces. Based on recent results in \cite{FNA2} we expect that oversampling together with suitable regularisation can act as a stabiliser towards redundancies in the basis spaces, which might provide a framework for rigorous analysis of collocation methods involving more general approximation spaces such as hybrid numerical-asymptotic basis functions as described by \cite{chandler-wilde_graham_langdon_spence_2012,gibbs2020}. A further direction for future research is the investigation of the merits of our current approach (relating a discrete sampling process to its continuous limit, and potentially the use of oversampling) to understand the convergence rates of fully discrete schemes such as Nyst\"{o}m methods \cite{bremer2012,bruno2013convergence,hao2014high}.

\section*{Acknowledgements}
The authors would like to thank Arieh Iserles (University of Cambridge) and Nigel Peake (University of Cambridge) for several stimulating discussions and feedback on the work. We would also like to express our gratitude to Evert Provoost (KU Leuven) for contributions to the code used in \S\ref{sec:numerical_results} and for performing initial numerical studies that helped inform this research.

\section*{Conflict of interest}
The authors declare that they have no conflict of interest.

% Bibliography
\bibliographystyle{siam}      
\bibliography{bibliography}

\begin{appendix}
\section{Error estimate in the discrete inner product}\label{app:proof_error_bound_trapezoidal_rule}
Here we describe the proof of Lemma~\ref{lem:bound_equispaced_trapezoidal_rule_L^2_product}. Since $\tilde{x}$ introduces a simple phase shift in all Fourier modes we may, without loss of generality, assume $\tilde{x}=0$. Let $\mathcal{Q}_M[f]=1/M\sum_{m=0}^{M-1}f(m/M)$ be the trapezoidal rule, then we have the following well-known result:
	\begin{lemma} If $f\in L^2([0,1))$ then
		\begin{align*}
			\int_0^1 f(x)dx-\mathcal{Q}_M[f]=\sum_{j\neq 0} \hat{f}_{jM}.
		\end{align*}
	\end{lemma}
	\begin{proof}
		\begin{align*}
			\mathcal{Q}_M[f]&=\frac{1}{M}\sum_{m=0}^{M-1}\sum_{k\in\mathbb{Z}}e^{2\pi i mk/M}\hat{f}_k=\frac{1}{M}\sum_{k\in\mathbb{Z}}\hat{f}_k\sum_{m=0}^{M-1}e^{2\pi i mk/M}=\sum_{j\neq 0}\hat{f}_{jM}.
		\end{align*}
	\end{proof}Thus we have immediately:
	\begin{corollary}\label{cor:quad_error_in_sobolev_norm}
	For all $t>1/2$, there is a constant $C_t>0$ (independent of $M$), such that
	\begin{align*}
		\left|\int_0^1 f(x)dx-\mathcal{Q}_M[f]\right|\leq C_t M^{-t}\|f\|_{H^t}.
	\end{align*}
	\end{corollary}
	In particular from Corollary~\ref{cor:quad_error_in_sobolev_norm} we have that for any $t>1/2$:
	\begin{align*}%\label{eqn:error_discrete_inner_product}
		\left|\langle f,g\rangle-\langle f,g\rangle_M\right|\leq C_t M^{-t}\|f\bar{g}\|_{H^{t}}.
	\end{align*}
\begin{proof}[Proof of Lemma~\ref{lem:bound_equispaced_trapezoidal_rule_L^2_product}]
	We observe (using the notation $\lesssim$ when there is a constant independent of $u,w$ implicit in the inequality):
	\begin{align}\nonumber
		\|uw\|_{H^{r}}^2&= \left|\sum_{m\in \mathbb{Z} }\hat{u}_{-m}\hat{w}_m\right|^2+\sum_{n\in \mathbb{Z} }|n|^{2r}\left|\sum_{m\in \mathbb{Z} }\hat{u}_{n-m}\hat{w}_m\right|^2\\\nonumber
		&\lesssim \left(\sum_{m\in \mathbb{Z} }|\hat{u}_{-m}|^2\right)\left(\sum_{m\in \mathbb{Z} }|\hat{w}_{m}|^2\right)+ \sum_{n\in \mathbb{Z} }\left|\sum_{m\in \mathbb{Z} }|n-m|^{r}\hat{u}_{n-m}\hat{w}_m+\sum_{m\in \mathbb{Z} }|m|^{r}\hat{u}_{n-m}\hat{w}_m\right|^2\\\label{eqn:split_into_two_terms}
		&\lesssim \|u\|_0^2\|w\|_{0}^2+ \sum_{n\in \mathbb{Z} }\left|\sum_{m\in \mathbb{Z} }|n-m|^{r}\hat{u}_{n-m}\hat{w}_m\right|^2+\left|\sum_{m\in \mathbb{Z} }|m|^{r}\hat{u}_{n-m}\hat{w}_m\right|^2
	\end{align}
where in \eqref{eqn:split_into_two_terms} we used that for $t\geq 0$ there is a constant $C=C_t>0$ such that
\begin{align*}
	(|m|+|n|)^t\leq C(|m|^t+|n|^t),\quad \forall m,n\in\mathbb{Z}.
\end{align*}
By the discrete Minkowski integral inequality
\begin{align}\label{eqn:proof_of_banach_algebra_sobolev_spaces_step2}
\sum_{n\in \mathbb{Z} }\left|\sum_{m\in \mathbb{Z} }|n-m|^{r}\hat{u}_{n-m}\hat{w}_m\right|^2&\leq \left(\sum_{m\in \mathbb{Z} }\left(\sum_{n\in \mathbb{Z} }|n-m|^{2r}|\hat{u}_{n-m}|^2|\hat{w}_m|^2\right)^{\frac{1}{2}}\right)^2\\\nonumber
&\leq \|u\|_{r}^2\left(\sum_{m\in \mathbb{Z} }|\hat{w}_m|\right)^2\lesssim_s \|u\|_r^2\|w\|_s^2,\quad \text{\ any\ }s>1/2.
\end{align}
The final term in \eqref{eqn:split_into_two_terms} can be bounded similarly and this concludes the proof of Lemma~\ref{lem:bound_equispaced_trapezoidal_rule_L^2_product}.
\end{proof}
\section{Error estimate for the discrete inner product non-uniform collocation points}\label{app:error_trapezoidal_rule_non-uniform_grid}
Here we provide a proof of Lemma~\ref{lem:error_estimate_trapezoidal_rule_general_sampling}.
\begin{proof}[Proof of Lemma~\ref{lem:error_estimate_trapezoidal_rule_general_sampling}]
    For $f\in C^2$ we have the well-known estimate
    \begin{align}
        \left|\int_{a}^bf(x)-\frac{b-a}{2}(f(b)+f(a))\right|\leq \frac{(b-a)^3}{12}f''(\xi), \quad \text{some\ } \xi\in(a,b).
    \end{align}
    This means in particular for any choice of quadrature points $0=x_1<x_2<\dots<x_M<1$ that the trapezoidal rule with weights,
    \begin{align*}
    \int_{0}^1f(x)dx\approx \sum_{m=1}^M\frac{|x_{j+1}-x_{j-1}|}{2}f(x_j),
\end{align*}
    has an error of the form
    \begin{align*}%\label{eqn:general_trapezoidal_error_estimate}
        \left|\int_{0}^1f(x)dx- \sum_{m=1}^M\frac{|x_{j+1}-x_{j-1}|}{2}f(x_j)\right|\leq \frac{M}{12}\sup_{x\in[0,1)}|f''(x)|\max_{1\leq j\leq N}|x_{j+1}-x_j|^3
    \end{align*}
    We can translate this to Sobolev spaces using Morrey's inequality: For any $l\in\mathbb{N}$
    \begin{align*}
       \left| f^{(l)}(x)\right|&=\left|\sum_{m\in\mathbb{Z}}(2\pi im)^le^{2\pi i m}\hat{f}_m \right|\leq \sum_{m\in\mathbb{Z}}|2\pi m|^l |\hat{f}_m|\\
       &\leq (2\pi)^l\left(\sum_{m\in\mathbb{Z}}|m|^{-2s} \right)^{1/2}\left(\sum_{m\in\mathbb{Z}}|m|^{2s+2l}|\hat{f}_m|^2\right)^{1/2}\lesssim_s \|f\|_{H^{l+s}}\quad \text{any\ }s>1/2,
    \end{align*}
    where $\lesssim_s$ indicates an implicit constant independent of $f$ but dependent on $s$. Thus we have for any $f\in H^r,r>5/2$,
    \begin{align*}
        \mathcal{E}\lesssim_r M\max_{1\leq j\leq N}|x_{j+1}-x_j|^3 \|f\|_{H^r}.
    \end{align*}
   Following through the same steps \eqref{eqn:split_into_two_terms} \& \eqref{eqn:proof_of_banach_algebra_sobolev_spaces_step2} as in Appendix~\ref{app:proof_error_bound_trapezoidal_rule} we find
    \begin{align*}
        |\langle f,g\rangle-\langle f,g\rangle_M|\lesssim_{r,s} M\max_{1\leq j\leq N}|x_{j+1}-x_j|^3\left(\|f\|_{H^r}\|g\|_{H^s}+\|f\|_{H^s}\|g\|_{H^r}\right),
    \end{align*}
    for any $r>5/2,s>1/2$.
\end{proof}
\section{Derivation of exact error expression for equispaced grids}\label{app:derivation_exact_error_expression}
Here we provide the derivation of \eqref{eqn:explicit_solution_linear_system} \& \eqref{eqn:exact_expression_error_equispaced_grids}. The arguments are analogous to the discussion in \cite[Section 2]{chandler1990} with very minor modificiations to adapt to our notation and the Bubnov-Galerkin setting. Given the pseudo-differential form \eqref{eqn:pseudo_differential_form_V} of $V$, its action on the basis $\psi_\mu$ is quickly determined to be
\begin{align*}
    V\psi_\mu&=\sum_{m\equiv\mu(N)}[m]^{2\alpha}\left(\frac{\mu}{m}\right)^{d+1}e^{2\pi i mx}=[m]^{2\alpha} \sum_{m\equiv \mu(N)}\left[\frac{m}{\mu}\right]^{2\alpha}\left(\frac{\mu}{m}\right)^{d+1}e^{2\pi i mx}\\
    &=[\mu]^{2\alpha} e^{2\pi i \mu x}\left(1+\Omega\left(Nx,\frac{\mu}{N}\right)\right),
\end{align*}
where
\begin{align*}
        \Omega(\xi,y)=|y|^{{d+1}-2\alpha}\sum_{l\neq 0}\frac{1}{|l+y|^{{d+1}-2\alpha}}e^{2\pi il\xi}.
\end{align*}
Thus we can write the discrete inner product \eqref{eqn:discrete_inner_product_equispaced_points} coming from our collocation points as follows
\begin{align*}
    \left\langle V\psi_{\mu}, V\psi_\nu\right\rangle_M=\begin{cases}0,&\text{if\ } \mu\neq \nu\\
    1,&\text{if\ }\mu=\nu=0\\
    [\mu]^{4\alpha}\frac{1}{J}\sum_{j=1}^J\left|1+\Omega\left(\xi_j,\frac{\mu}{N}\right)\right|^2,&\text{if\ }\mu=\nu\neq 0.
    \end{cases}
\end{align*}
Similarly, we can compute
\begin{align*}
\left\langle V\psi_\mu,\exp\left(2\pi in\, \cdot\,\right)\right\rangle_M=\begin{cases}0,&\text{if\ }\mu\not\equiv \nu(N),\\
\frac{1}{J}\sum_{j=1}^J \exp(2\pi i l\xi_j),&\text{if\ }n=lN,\mu=0,\\
\frac{1}{J}\sum_{j=1}^J\exp(2\pi i l\xi_j)[\mu]^{2\alpha}\left(1+\overline{\Omega\left(\xi_j,\frac{\mu}{N}\right)}\right),&\text{if\ }n=\mu+lN,\mu\neq 0.
\end{cases}
\end{align*}
Thus we have for a general $u$:
\begin{align*}
     \left\langle V\psi_{\mu},Vu\right\rangle_M&=\sum_{m\in\mathbb{Z}}[m]^{2\alpha}\hat{u}_m\left\langle\exp\left(2\pi im\, \cdot\,\right),A\psi_\mu\right\rangle_M\\
     &=\begin{cases}
     \frac{1}{J}\sum_{j=1}^J\sum_{n\equiv 0(N)}[n]^{2\alpha}\hat{u}_n\exp\left(2\pi i \frac{n}{N} \xi_j\right),&\text{if\ }\mu=0,\\
     [\mu]^{2\alpha}\frac{1}{J}\sum_{j=1}^J\sum_{n\equiv \mu(N)}\exp\left(2\pi i \frac{n-\mu}{N} \xi_j\right)[n]^{2\alpha}\hat{u}_n\left(1+\overline{\Omega\left(\xi_j,\frac{\mu}{N}\right)}\right),&\text{if\ }\mu\neq0.
     \end{cases}
\end{align*}
Hence the linear system \eqref{eqn:discrete_orthogonality_special_basis} for the coefficients $a_\mu$ of $u_N^{(M)}$ in the basis $\psi_\mu$ ($u^{(M)}_N=\sum_{\nu\in\Lambda_N}\hat{a}_\nu \psi_{\nu}$) is diagonal and we find
\begin{align}\label{eqn:explicit_solution_linear_system_appendix}
    a_\mu=\begin{cases}\frac{1}{J}\sum_{j=1}^J\sum_{n\equiv 0(N)}[n]^{2\alpha}\hat{u}_n\exp\left(\frac{n}{N}\xi_j\right),&\text{if\ }\mu=0\\
    D\left(\frac{\mu}{N}\right)^{-1}\frac{1}{J}\sum_{j=1}^J\sum_{n\equiv\mu(N)}\left[\frac{n}{\mu}\right]^{2\alpha}\exp\left(2\pi i \frac{n-\mu}{N}\xi_j\right) \hat{u}_n\left(1+\overline{\Omega\left(\xi_j,\frac{\mu}{N}\right)}\right),&\text{if\ }\mu\neq 0,
    \end{cases}
\end{align}
where
\begin{align*}
    D(y)=\frac{1}{J}\sum_{j=1}^J\left|1+\Omega\left(\xi_j,y\right)\right|^2.
\end{align*}
As in \S\ref{sec:exact_expression_error_equispaced_splines} we let the true solution to \eqref{eqn:original_integral_equation} be $\tilde{u}(x)=\sum_{m\in\mathbb{Z}}\hat{u}_m\exp(2\pi im x)$. Thus, simplifying \eqref{eqn:explicit_solution_linear_system_appendix} we find after a few steps of algebra the required expressions \eqref{eqn:exact_expression_error_equispaced_grids}:
\begin{align*}
  a_\mu-\hat{u}_\mu=\begin{cases}P_N,&\text{if\ }\mu=0\\
  -\frac{E(\mu/N)}{D(\mu/N)}\hat{u}_\mu+R_N(\mu),&\text{if\ }\mu\neq 0,
  \end{cases}
\end{align*}
where:
\begin{align*}
    P_N&=\frac{1}{J}\sum_{j=1}^J\sum_{\substack{n\equiv 0(N)\\n\neq 0}}[n]^{2\alpha}\hat{u}_n\exp\left(2\pi i n\xi_j/N\right)\\
    E(y)&=\frac{1}{J}\sum_{j=1}^J\Omega\left(\xi_j,y\right)\left(1+\overline{\Omega\left(\xi_j,y\right)}\right)\\
    R_N(\mu)&=D\left(\frac{\mu}{N}\right)^{-1}\frac{1}{J}\sum_{j=1}^J\sum_{\substack{n\equiv\mu(N)\\n\neq \mu}}\left[\frac{n}{\mu}\right]^{2\alpha}\exp\left(2\pi i\frac{n-\mu}{N}\xi_j\right) \hat{u}_n\left(1+\overline{\Omega\left(\xi_j,\frac{\mu}{N}\right)}\right).
\end{align*}
We can now use the fact that $\xi_j=j/J$ and the identity
\begin{align}\label{eqn:trig_identity}
    \frac{1}{J}\sum_{j=1}^J\exp(2\pi i m j/J)=\begin{cases}1,&\quad m\equiv 0 (J)\\
    0,&\quad m\not\equiv 0 (J)
    \end{cases}
\end{align}
to further simplify the above expressions:
\begin{align}\nonumber
    D(y)&=1+\frac{1}{J}\sum_{j=1}^J\left|\Omega\left(\xi_j,y\right)\right|^2+2\Re\left(\frac{1}{J}\sum_{j=1}^J\Omega\left(\xi_j,y\right)\right)\\\nonumber
    &=1+\frac{1}{J}\sum_{j=1}^J\left|\Omega\left(\xi_j,y\right)\right|^2+2|y|^{{d+1}-2\alpha}\sum_{l\neq 0}\frac{1}{|lJ+y|^{{d+1}-2\alpha}}\\\label{eqn:app_proof_that_D_uniformly_bounded_away_from_0}
    &\geq 1, \quad \forall y\in[-1/2,1/2]
\end{align}
and similarly we find
\begin{align*}
    Z_N&=\sum_{\substack{n\in\mathbb{Z}\\n\neq 0}}\left[nJN\right]^{2\alpha}\hat{u}_{nJN}=\sum_{\substack{n\in\mathbb{Z}\\n\neq 0}}\left[nM\right]^{2\alpha}\hat{u}_{nM}\\
    E(y)&=|y|^{d+1-2\alpha}\sum_{l\neq0}\frac{1}{|lJ+y|^{d+1-2\alpha}}+\frac{1}{J}\sum_{j=1}^J\left|\Omega(\xi_j,y)\right|^2,
\end{align*}
where we made extensive use of the trigonometric identity \eqref{eqn:trig_identity}. Finally, we simplify the expression for $R_N(\mu)$
\begin{align*}
    R_N(\mu)&=D\left(\frac{\mu}{N}\right)^{-1}\left(\sum_{k\neq 0}\left[\frac{\mu+kM}{\mu}\right]^{2\alpha}\hat{u}_{\mu+kM}\right.\\
    &\quad+\left.\sum_{\substack{n\equiv\mu(N)\\n\neq \mu}}\left[\frac{n}{\mu}\right]^{2\alpha}\hat{u}_n\left|\frac{\mu}{N}\right|^{d+1-2\alpha}\frac{1}{J}\sum_{j=1}^J\sum_{l\neq 0}\frac{1}{|l+\mu/N|^{d+1-2\alpha}}\exp\left(2\pi i\left(\frac{n-\mu}{N}-l\right)\frac{j}{J}\right)\right)\\
   &=D\left(\frac{\mu}{N}\right)^{-1}\left(\sum_{k\neq 0}\left[\frac{\mu+kM}{\mu}\right]^{2\alpha}\hat{u}_{\mu+kM}\right.\\
   &\quad\quad\left.+\sum_{k\neq 0}\left[\frac{\mu+kN}{\mu}\right]^{2\alpha}\hat{u}_{\mu+kN}\left|\frac{\mu}{N}\right|^{d+1-2\alpha}\sum_{\substack{l\equiv k (J)\\l\neq 0}}\left|\frac{1}{l+\mu/N}\right|^{d+1-2\alpha}\right).
\end{align*}

\section{A compact perturbation argument}\label{app:compact_perturbation_argument}
Here we outline a brief overview of the compact perturbation argument referenced in Remark~\ref{rmk:modified_oversampled_col_and_compact_perturbation_argument} that is standard in the relevant literature for the analysis of collocation methods, see for instance \S3 in \cite{Arnold1985}. Suppose we proved the following apriori estimate for some $s<t$: If for some $a_n\in S_N$ and some $\tilde{a}\in H^{t}$:
\begin{align}\label{eqn:unperturbed_apriori_estimate}
    \left\langle\chi_N,V_0a_N\right\rangle_M=\left\langle\chi_N,V_0\tilde{a}\right\rangle_M\quad\forall \chi_N\in S_N\quad\text{then}\quad\|\tilde{a}-a_N\|_{s}\lesssim N^{s-t}\|\tilde{a}\|_t.
\end{align}
Under the compact perturbation $V=V_0+\mathcal{K}$ where $\mathcal{K}:H^{p}\rightarrow H^q$ is continuous for all $p,q\in\mathbb{R}$ the perturbed linear conditions become
\begin{align}\label{eqn:perturbed_linear_conditions}
    \left\langle\chi_N,(V_0+\mathcal{K})u_N\right\rangle_M=\left\langle\chi_N,(V_0+\mathcal{K})\tilde{u}\right\rangle_M\quad\forall \chi_N\in S_N.
\end{align}
The trick given by \cite{Arnold1985} is then to write \eqref{eqn:perturbed_linear_conditions} in the equivalent form
\begin{align*}%\label{eqn:alternative_form_perturbed_linear_conditions}
    \left\langle\chi_N,V_0u_N\right\rangle_M=\left\langle\chi_N,V_0\left(V_0^{-1}(V_0+\mathcal{K})\tilde{u}-V_0^{-1}\mathcal{K} u_N\right)\right\rangle_M\quad\forall \chi_N\in S_N,
\end{align*}
which means that by \eqref{eqn:unperturbed_apriori_estimate} we have
\begin{align*}
   \| \left(V_0^{-1}(V_0+\mathcal{K})\tilde{u}-V_0^{-1}\mathcal{K} u_N\right)-u_N\|_{s}\lesssim N^{s-t}\|V_0^{-1}(V_0+\mathcal{K})\tilde{u}-V_0^{-1}\mathcal{K} u_N\|_t.
\end{align*}
Simplifying both sides we have
\begin{align*}
     \| V_0^{-1}(V_0+\mathcal{K})\left(\tilde{u}-u_N\right)\|_{s}\lesssim N^{s-t}\|\tilde{u}+V_0^{-1}\mathcal{K}\left(\tilde{u}- u_N\right)\|_t.
\end{align*}
Now by the pseudo-differential form \eqref{eqn:pseudo_differential_form_V} $V_0$ is continuous and by assumption $V=V_0+\mathcal{K}$ is invertible, thus we have
\begin{align*}
   \|\tilde{u}-u_N\|_{s} \lesssim_s \| V_0^{-1}(V_0+\mathcal{K})\left(\tilde{u}-u_N\right)\|_{s}
\end{align*}
and by compactness of $\mathcal{K}$ and continuity of $V_0^{-1}$
\begin{align*}
    \|\tilde{u}+V_0^{-1}\mathcal{K}\left(\tilde{u}- u_N\right)\|_t\leq \|\tilde{u}\|_{t}+ C\|\tilde{u}- u_N\|_{s}.
\end{align*}
Thus we have overall
\begin{align*}
    \|\tilde{u}-u_N\|_{s} (1-CN^{s-t})\lesssim N^{s-t}\|\tilde{u}\|_{t}.
\end{align*}
Finally, since $(1-CN^{s-t})\rightarrow 1$ as $N\rightarrow\infty$, we find for some $N_0>0$ and all $N\geq N_0$
\begin{align*}
    \|\tilde{u}-u_N\|_{s} \lesssim N^{s-t}\|\tilde{u}\|_{t}.
\end{align*}
\end{appendix}
\end{document}